\documentclass[11pt, a4paper, oneside]{article}
\usepackage[utf8]{inputenc}    % allow utf-8 input
\usepackage{cmap}              % (or mmap)
\usepackage[T1]{fontenc}       % use 8-bit T1 fonts
\usepackage{lmodern} 
\usepackage{natbib}
\usepackage[colorlinks=true]{hyperref}
\usepackage{amssymb, amsthm, amsmath, amsfonts, verbatim}
\usepackage{authblk, mathtools, longfbox, dsfont, mathrsfs, color, bm, enumitem, url, upgreek}
\usepackage{calc}
\usepackage{graphicx, transparent, wrapfig, subcaption, caption}
\usepackage{titlesec}
\usepackage{algpseudocode}
\usepackage{multicol}
\usepackage{tcolorbox}
\usepackage{algorithm}
\usepackage{todonotes}
\titlelabel{\thetitle.\;}
\DeclareMathOperator*{\argmin}{argmin}

\newcommand{\R}{\mathbb{R}}
\newcommand{\mc}{\mathcal}
\newcommand{\DS}{\displaystyle}
\newcommand{\st}{\mathrm{s.t.}}

\newcommand{\mbb}{\mathbb}

\graphicspath{{figures/}}
\theoremstyle{definition}
\newtheorem{theorem}{Theorem}[section]

\newtheorem{corollary}[theorem]{Corollary}
\newtheorem{lemma}[theorem]{Lemma}
\newtheorem{remark}[theorem]{Remark}
\newtheorem{example}[theorem]{Example}
\usepackage{xcolor}
\hypersetup{colorlinks, linkcolor={red!60!black}, citecolor={blue!35!black}, urlcolor={green!50!black}}

\usepackage[a4paper, total={7in, 10in}]{geometry}
\usepackage{accents}

\usepackage{lineno}
\hypersetup{breaklinks=true}% allow pdf compability

\allowdisplaybreaks
\author[$\dagger$]{Bahar Ta{\c{s}}kesen}
\author[$\ddagger$]{Soroosh Shafieezadeh-Abadeh}
\author[$\dagger$]{Daniel Kuhn}
\author[$\mathsection$]{Karthik Natarajan}

\affil[$\dagger$]{Risk Analytics and Optimization Chair, EPFL Lausanne \authorcr \texttt{bahar.taskesen,daniel.kuhn@epfl.ch}}
\affil[$\ddagger$]{Tepper School of Business, CMU \authorcr \texttt{sshafiee@andrew.cmu.edu}}
\affil[$\mathsection$]{{Engineering Systems and Design,
Singapore University of Technology and Design} \authorcr\texttt{karthik$\_$natarajan@sutd.edu.sg}}

\date{}

\begin{document}
\title{Discrete Optimal Transport with Independent Marginals is $\#$\textbf{P}-Hard}
% other possibilities
% Hardness of discrete optimal transport with independent random vectors 
% Discrete optimal transport problems can be hard

\maketitle

\begin{abstract}

    We study the computational complexity of the optimal transport problem that evaluates the Wasserstein distance between the distributions of two $K$-dimensional discrete random vectors. The best known algorithms for this problem run in polynomial time in the maximum of the number of atoms of the two distributions. However, if the components of either random vector are independent, then this number can be exponential in~$K$ even though the size of the problem description scales linearly with~$K$. We prove that the described optimal transport problem is $\#$\textbf{P}-hard even if all components of the first random vector are independent uniform Bernoulli random variables, while the second random vector has merely two atoms, and even if only approximate solutions are sought. We also develop a dynamic programming-type algorithm that approximates the Wasserstein distance in pseudo-polynomial time when the components of the first random vector follow arbitrary independent discrete distributions, and we identify special problem instances that can be solved exactly in strongly polynomial time.
\end{abstract}
% \begin{keywords}
%   optimal transport, hardness, computational complexity, $\#$\textbf{P}-hardness
% \end{keywords}

\section{Introduction}
Optimal transport theory is closely intertwined with probability theory and statistics~\citep{boucheron2013concentration, villani} as well as with economics and finance~\citep{galichon2016optimal}, and it has spurred fundamental research on partial differential equations~\citep{benamou2000computational, brenier1991polar}. In addition, optimal transport problems naturally emerge in numerous application areas spanning machine learning \citep{arjovsky2017wasserstein, carriere2017sliced, rolet2016fast}, signal processing~\citep{ferradans2014regularized, kolouri2015transport, papadakis2017convex, tartavel2016wasserstein}, computer vision~\citep{rubner2000earth, solomon2014earth, solomon2015convolutional} and distributionally robust optimization~\citep{blanchet2019quantifying, gao2016distributionally, esfahani2018data}. For a comprehensive survey of modern applications of optimal transport theory we refer to~\citep{kolouri2017optimal, peyre2019computational}. Historically, the first optimal transport problem was formulated by
Gaspard Monge as early as in~1781 \citep{monge1781memoire}. Monge's formulation aims at finding a measure-preserving map that minimizes some notion of transportation cost between two probability distributions, where all probability mass at a given origin location must be transported to the same target location. 
Due to this restriction, an optimal transportation map is not guaranteed to exist in general, and Monge's problem could be infeasible. In 1942, Leonid Kantorovich formulated a convex relaxation of Monge's problem by introducing the notion of a transportation plan that allows for mass splitting~\citep{kantorovich1942transfer}. Interestingly, an optimal transportation plan always exists. This paradigm shift has served as a catalyst for significant progress in the field.

In this paper we study Kantrovich's optimal transport problem between two discrete distributions
\[
    \mu = \sum_{i \in \mathcal I} \mu_i \delta_{\bm x_i}\quad \text{and}\quad \nu = \sum_{j \in \mathcal J} \nu_j \delta_{\bm y_j},
\]
on $\mathbb R^K$, where $\bm \mu\in\mathbb R^I$ and $\bm \nu\in\mathbb R^J$ denote the probability vectors, whereas $\bm x_i \in \mathbb R^K$ for $i \in \mathcal I = \{1, \dots, I\}$ and $\bm y_j \in \mathbb R^K$ for $j \in \mathcal J = \{1, \dots, J\}$ represent the discrete support points of~$\mu$ and~$\nu$, respectively. Throughout the paper we assume that $\mu$ and $\nu$ denote the probability distributions of two $K$-dimensional discrete random vectors $\bm x$ and $\bm y$, respectively.
Given a transportation cost function $c: \mathbb R^K \times \mathbb R^K \to [0,+\infty]$, we define the optimal transport distance between the discrete distributions~$\mu$ and~$\nu$ as
%{\color{black} and provided that the input descriptions of $\mu$ and $\nu$ scale polynomially in~$K$, we define the discrete optimal transport distance between $\mu$ and $\nu$ as}
\begin{equation}
W_c(\mu, \nu) = \min\limits_{\bm \pi \in \Pi(\bm \mu,\bm \nu)} ~ \sum_{i \in \mathcal I} \sum_{j \in \mathcal J} c(\bm x_i, \bm y_j) \pi_{ij},
\label{eq:primal}
\end{equation}
where $\Pi(\bm \mu,\bm \nu) = \{ \bm \pi \in \mathbb R^{I \times J}_+: \bm \pi \bm 1 = \bm \mu,\; \bm \pi^\top \bm 1 = \bm \nu \}$ denotes the polytope of probability matrices~$\bm \pi$ with marginal probability vectors $\bm \mu$ and $\bm \nu$. Thus, every $\bm \pi \in \Pi(\bm \mu, \bm \nu)$ defines a discrete probability distribution
\[
    \pi = \sum_{i \in \mathcal I} \sum_{j \in \mathcal J} \pi_{ij} \delta_{(\bm x_i, \bm y_j)}
\]
of~$(\bm x,\bm y)$ under which~$\bm x$ and~$\bm y$ have marginal distributions $\mu$ and $\nu$, respectively. Distributions with these properties are referred to as transportation plans. If there exists $p \geq 1$ such that~$c(\bm x, \bm y)= \|\bm x- \bm y\|^p$ for all~$\bm x, \bm y\in\mathbb R^K$, then $W_c(\mu, \nu)^{1/p}$ is termed the $p$-th Wasserstein distance between $\mu$ and $\nu$. The optimal transport problem~\eqref{eq:primal} constitutes a linear program that admits a strong dual linear program of the form
\begin{equation*}
\begin{array}{c@{\quad}l@{\qquad}l}
\max & \bm \mu^\top \bm \psi + \bm \nu^\top \bm \phi & \\ [0.5em]
\st & \bm \psi \in \mathbb R^I, ~ \bm \phi \in \mathbb R^J & \\ [0.5em]
& \psi_i + \phi_j \leq c(\bm x_i, \bm y_j) \quad \forall i \in \mathcal I, j \in \mathcal J.
\end{array}
\end{equation*}
Strong duality holds because $\bm \pi = \bm \mu \bm \nu^\top$ is feasible in~\eqref{eq:primal} and the optimal value is finite. Both the primal and the dual formulations of the optimal transport problem can be solved exactly using the simplex algorithm~\citep{dantzig1951application}, the more specialized network simplex algorithm~\citep{orlin1997polynomial} or the Hungarian algorithm~\citep{kuhn1955hungarian}. Both problems can also be addressed with dual ascent methods~\citep{bertsimas1997introduction}, customized auction algorithms \citep{bertsekas1981new, bertsekas1992auction} or interior point methods~\citep{karmarkar1984new, lee2014path, nesterov1994interior}.
More recently, the emergence of high-dimensional optimal transport problems in machine learning has spurred the development of efficient approximation algorithms. Many popular approaches for approximating the optimal transport distance between two discrete distributions rely on solving a regularized variant of problem~\eqref{eq:primal}. For instance, when augmented with an entropic regularizer, problem~\eqref{eq:primal} becomes amenable to greedy methods such as the Sinkhorn algorithm~\citep{sinkhorn1967diagonal, cuturi2013sinkhorn} or the related Greenkhorn algorithm~\citep{abid2018greedy, altschuler2017near, chakrabarty2018better}, which run orders of magnitude faster than the exact methods. Other promising regularizers that have attracted significant interest include the Tikhonov~\citep{blondel2017smooth, dessein2018regularized, essid2018quadratically, seguy2017large}, Lasso~\citep{li2016fast}, Tsallis entropy~\citep{muzellec2017tsallis} and group Lasso regularizers~\citep{courty2016optimal}. In addition, Newton-type methods~\citep{blanchet2018towards, quanrud2018approximating}, quasi-Newton methods \citep{blondel2017smooth}, primal-dual gradient methods \citep{dvurechensky2018computational, guo2020fast, jambulapati2019direct, lin2019efficient, lin2019acceleration}, iterative Bregman projections \citep{benamou2015iterative} and stochastic average gradient descent algorithms \citep{genevay2016stochastic} are also used to find approximate solutions for discrete optimal transport problems. 

The existing literature mainly addresses optimal transport problems between discrete distributions that are specified by enumerating the locations and the probabilities of the underlying atoms. In this case, the worst-case time-complexity of solving the linear program~\eqref{eq:primal} with an interior point algorithm, say, grows polynomially with the problem's input description. %number of atoms of the discrete distributions $\mu$ and $\nu$. 
In contrast, we focus here on optimal transport problems between discrete distributions supported on a number of points that grows {\em exponentially} with the dimension~$K$ of the sample space even though these problems admit an input description that scales only {\em polynomially} with~$K$. In this case, the worst-case time-complexity of solving the linear program~\eqref{eq:primal} directly with an interior point algorithm grows exponentially with the problem's input description. More precisely, we henceforth assume that~$\mu$ is the distribution of a random vector~$\bm x\in\mathbb R^K$ with independent components. Hence, $\mu$ is uniquely determined by the specification of its $K$ marginals, which can be encoded using~$\mathcal O(K)$ bits. Yet, even if each marginal has only two atoms, $\mu$~accommodates already $2^K$~atoms. Optimal transport problems involving such distributions are studied by \cite{ccelik2021wasserstein} with the aim to find a discrete distribution with independent marginals that minimizes the Wasserstein distance from a prescribed discrete distribution. While~\cite{ccelik2021wasserstein} focus on solving small instances of this nonconvex problem, our results surprisingly imply that even evaluating this problem's objective function is hard. In summary, we are interested in scenarios where  
the discrete optimal transport problem~\eqref{eq:primal} constitutes a linear program with exponentially many variables and constraints. We emphasize that such linear programs are not necessarily hard to solve \citep{grotschel2012geometric}, and therefore a rigorous complexity analysis is needed. We briefly review some useful computational complexity concepts next.

Recall that the complexity class \textbf{P} comprises all decision problems ({\em i.e.}, problems with a Yes/No answer) that can be solved in polynomial time. In contrast, the complexity class \textbf{NP} comprises all decision problems with the property that each `Yes' instance admits a certificate that can be verified in polynomial time. A problem is \textbf{NP}-hard if every problem in \textbf{NP} is polynomial-time reducible to it, and an \textbf{NP}-hard problem is \textbf{NP}-complete if it belongs to \textbf{NP}. In this paper we will mainly focus on the complexity class $\#$\textbf{P}, which encompasses all counting problems associated with decision problems in~\textbf{NP}~\citep{valiant1979complexity,valiant1979complexitypermanent}. Loosely speaking, an instance of a $\#$\textbf{P} problem thus counts the number of distinct polynomial-time verifiable certificates of the corresponding~\textbf{NP} instance. Consequently, a $\#$\textbf{P} problem is at least as hard as its \textbf{NP} counterpart, and~$\#$\textbf{P} problems cannot be solved in polynomial time unless $\#$\textbf{P} coincides with the class~\textbf{FP} of polynomial-time solvable function problems. A Turing reduction from a function problem~$A$ to a function problem~$B$ is an algorithm for solving problem~$A$ that has access to a fictitious oracle for solving problem~$B$ in one unit of time. Note that the oracle plays the role of a subroutine and may be called several times. A polynomial-time Turing reduction from~$A$ to~$B$ runs in time polynomial in the input size of~$A$. We emphasize that, even though each oracle call requires only one unit of time, the time needed for computing all oracle inputs and reading all oracle outputs is attributed to the runtime of the Turing reduction. A problem is $\#$\textbf{P}-hard if every problem in $\#$\textbf{P} is polynomial-time Turing reducible to it, and a $\#$\textbf{P}-hard problem is $\#$\textbf{P}-complete if it belongs to $\#$\textbf{P} \citep{valiant1979complexitypermanent, jerrum2003counting}.

%When confronted with a new problem, a natural question is whether it can be solved efficiently in polynomial time. Even if it is proved that this problem belongs to a hard problem class that is believed to be unsolvable in polynomial time, it is still useful to study special instances of the hard problem because many polynomial time solvable problems differ only by little from the corresponding hard problems~\citep[\S~4]{garey1979computers}. Indeed, by restricting the problem inputs further of the original hard problem  sometimes one can design a pseudo-polynomial time algorithm that would serve as good as a polynomial time algorithm whenever the values of the input parameters does not have exponentially large values. Then, one can discuss the existence of a pseudo-polynomial time algorithm for the subproblems of the original hard problem. If there exists no such algorithm, the hardness result hold in a strong sense, \textit{see} for example \textit{strongly} \textbf{NP}-complete problems~\cite[Obervation~4.2]{garey1979computers}. 

Several hardness results for variants and generalizations of the optimal transport problem have recently been discovered. For example, multi-marginal optimal transport and Wasserstein barycenter problems were shown to be \textbf{NP}-hard \citep{altschuler2020hardness, altschuler2021wasserstein}, whereas the problem of computing the Wasserstein distance between a continuous and a discrete distribution was shown to be $\#$\textbf{P}-hard even in the simplest conceivable scenarios~\citep{taskesen2021semi}. In this paper, we focus on optimal transport problems between two discrete distributions~$\mu$ and~$\nu$. We formally prove that such problems are also $\#$\textbf{P}-hard when~$\mu$ and/or~$\nu$ may have independent marginals. Specifically, we establish a fundamental limitation of all numerical algorithms for solving optimal transport problems between discrete distributions~$\mu$ and~$\nu$, where~$\mu$ has independent marginals. We show that, unless~$\textbf{FP}=\#\textbf{P}$, it is not possible to design an algorithm that approximates~$W_c(\mu,\nu)$ in time polynomial in the bit length of the input size (which scales only polynomially with the dimension~$K$) and the bit length~$\log_2(1/\varepsilon)$ of the desired accuracy~$\varepsilon>0$. This result prompts us to look for algorithms that output $\varepsilon$-approximations in {\em pseudo-polynomial time}, that is, in time polynomial in the input size, the magnitude of the largest number in the input and the inverse accuracy~$1/\varepsilon$. It also prompts us to look for special instances of the optimal transport problem with independent marginals that can be solved in {\em weakly} or {\em strongly polynomial time}. An algorithm runs in weakly polynomial time if it computes~$W_c(\mu,\nu)$ in time polynomial in the bit length of the input. Similarly, an algorithm runs in strongly polynomial time if it computes~$W_c(\mu,\nu)$ in time polynomial in the bit length of the input and if, in addition, it requires a number of arithmetic operations that grows at most polynomially with the dimension of the input ({\em i.e.}, the number of input numbers).

The key contributions of this paper can be summarized as follows.
\begin{itemize}[leftmargin=6mm]
    \item We prove that the discrete optimal transport problem with independent marginals is $\#$\textbf{P}-hard even if~$\mu$ represents the uniform distribution on the vertices of the $K$-dimensional hypercube and $\nu$ has only two support points, and even if only approximate solutions of polynomial bit length are sought (see Theorem \ref{theorem:approx-hard}).
    
    \item We demonstrate that the discrete optimal transport problem with independent marginals can be solved in strongly polynomial time by a dynamic programming-type algorithm if both~$\mu$ and~$\nu$ are supported on a fixed bounded subset of a scaled integer lattice with a fixed scaling factor {\color{black} and if~$\nu$ has only two atoms}---even if~$\mu$ represents an arbitrary distribution with independent marginals (see Theorem~\ref{theorem:wass_app_complexity}, Corollary~\ref{corollary:tractable_example0} and the subsequent discussion). The design of this algorithm reveals an intimate connection between optimal transport and the conditional value-at-risk arising in risk measurement.
    
    \item Using a rounding scheme to approximate~$\mu$ and~$\nu$ by distributions~$\tilde\mu$ and~$\tilde\nu$ supported on a scaled integer lattice with a sufficiently small grid spacing constant, we show that {\color{black} if~$\nu$ has only two atoms, then} $\varepsilon$-accurate approximations of the optimal transport distance between~$\mu$ and~$\nu$ can always be computed in pseudo-polynomial time via dynamic programming{\color{black}---even if~$\mu$ represents an arbitrary distribution with independent marginals} (see Theorem~\ref{theorem:approximation}). This result implies that the optimal transport problem with independent marginals is in fact $\#$\textbf{P}-hard in the weak sense~\citep[\S~4]{garey1979computers}.
    
\end{itemize}

Our complexity analysis complements existing hardness results for two-stage stochastic programming problems. Indeed, \cite{dyer2006, dyer2015}, \cite{Grani} and \cite{dhara2021} show that computing optimal first-stage decisions of linear two-stage stochastic programs and evaluating the corresponding expected costs is hard if the uncertain problem parameters follow independent (discrete or continuous) distributions. This paper establishes similar hardness results for discrete optimal transport problems. Our paper also complements the work of \citet{genevay2016stochastic}, who describe a stochastic gradient descent method for computing $\varepsilon$-optimal transportation plans in $\mathcal O(1/\varepsilon^2)$ iterations. Their method can in principle be applied to the discrete optimal transport problems with independent marginals studied here. However, unlike our pseudo-polynomial time dynamic programming-based algorithm, their method is non-deterministic and does not output an approximation of the optimal transport distance~$W_c(\mu, \nu)$.

The remainder of this paper is structured as follows. In Section~\ref{sec:knapsack} we review a useful $\#$\textbf{P}-hardness result for a counting version of the knapsack problem. By reducing this problem to the optimal transport problem with independent marginals, we
prove in Section~\ref{sec:complexity} that the latter problem is also $\#$\textbf{P}-hard even if only approximate solutions are sought. In Section~\ref{sec:polynomial} we develop a dynamic programming-type algorithm that computes approximations of the optimal transport distance in pseudo-polynomial time, and we identify special problem instances that can be solved exactly in strongly polynomial time.

\paragraph{Notation.}
We use boldface letters to denote vectors and matrices.
The vectors of all zeros and ones are denoted by $\bm 0$ and $\bm 1$, respectively, and their dimensions are always clear from the context.
The calligraphic letters $\mathcal I, \mathcal J, \mathcal K$ and $\mathcal L$ are reserved for finite index sets with cardinalities $I, J, K$ and $L$, that is, $\mathcal I = \{ 1, \dots, I \}$ etc. We denote by $\|\cdot\|$ the 2-norm, and for any~$\bm x\in\mathbb R^K$ we use $\delta_{\bm x}$ to denote the Dirac distribution at $\bm x$.

\section{A Counting Version of the Knapsack Problem}
\label{sec:knapsack}
Counting the number of feasible solutions of a $0/1$ knapsack problem is a seemingly simple but surprisingly challenging task. Formally, the problem of interest is stated as follows.
\begin{center}
\fbox{\parbox{0.825\columnwidth}{ {\centering
\vspace{0.5ex}\textsc{$\#$Knapsack}\\[1ex]}
		\textbf{Instance.} A list of items with weights $w_k\in\mathbb Z_+$, $k \in \mathcal K$, and a capacity $b\in\mathbb Z_+$. \\[1ex]
		\textbf{Goal.} Count the number of subsets of the items whose total weight is at most $b$. \vspace{0.5ex}
	}}
\end{center}
The \textsc{$\#$Knapsack} problem is known to be $\#$\textbf{P}-complete~\citep{dyer1993mildly}, and thus it admits no polynomial-time algorithm unless~$\textbf{FP}=\#\textbf{P}$. \cite{dyer1993mildly} discovered a randomized sub-exponential time algorithm that provides almost correct solutions with high probability by sampling feasible solutions using a random walk. By relying on a rapidly mixing Markov chain, \cite{morris2004random} then developed the first fully polynomial randomized approximation scheme. Later, \cite{dyer2003approximate} interweaved dynamic programming and rejection sampling approaches to obtain a considerably simpler fully polynomial randomized approximation scheme. However, randomization remains essential in this approach. Deterministic dynamic programming-based algorithms were developed more recently by \cite{gopalan2011fptas}, and  \cite{vstefankovivc2012deterministic}. In the next section we will demonstrate that a certain class of discrete optimal transport problems with independent marginals is at least as hard as the \textsc{$\#$Knapsack} problem.

\section{Optimal Transport with Independent Marginals}
\label{sec:complexity}
Consider now a variant of the optimal transport problem~\eqref{eq:primal}, where the discrete multivariate distribution $\mu = \otimes_{k \in \mathcal K} \mu_k$ is a product of~$K$ independent univariate marginal distributions~$\mu_k = \sum_{l \in \mathcal L} \mu_k^l \delta_{x_k^l}$ with support points~$x_k^l \in \mathbb R$ and corresponding probabilities~$\mu_k^l$ for every~$l \in \mathcal L$. This implies that~$\mu$ accommodates a total of~$I = L^K$ support points. The assumption that each~$\mu_k$, $k\in\mathcal K$, accommodates the same number~$L$ of support points simplifies notation but can be imposed without loss of generality. Indeed, the probability of any unneeded support point can be set to zero. The other discrete multivariate distribution~$\nu= \sum_{j \in \mathcal J} \nu_j \delta_{\bm y_j}$ has no special structure.
Assume for the moment that all components of the support points as well as all probabilities of~$\mu_k$, $k\in\mathcal K$, and~$\nu$ are rational numbers and thus representable as ratios of two integers, and denote by~$U$ the maximum absolute numerical value among all these integers, which can be encoded using~$\mathcal O(\log_2 U)$ bits. Thus, the total number of bits needed to represent the discrete distributions~$\mu$ and~$\nu$ is bounded above by $\mathcal O(\max \{ K L, J \} \log_2 U)$. 
Note that this encoding does {\em not} require an explicit enumeration of the locations and probabilities of the~$I = L^K$ atoms of the distribution~$\mu$. It is well known that the linear program~\eqref{eq:primal} can be solved in polynomial time by the ellipsoid method, for instance, if~$\mu$ is encoded by such an inefficient exhaustive enumeration, which requires up to~$\mathcal O (\max\{I,J\}\log_2U)$ input bits. Thus, the runtime of the ellipsoid method scales at most polynomially with~$I$, $J$ and~$\log_2U$. As~$I = L^K$ grows exponentially with~$K$, however, this does {\em not} imply tractability of the optimal transport problem at hand, which admits an efficient encoding that scales only linearly with~$K$.
%In prior work, the polynomial-time complexity results for the discrete optimal transport problem are derived under the assumption that the input size is given by $\mathcal O (\max(I,J) \log_2U)$, which is here exponential in~$K$. 
In the remainder of this paper we will prove that the optimal transport problem with independent maringals is $\#$\textbf{P}-hard, and we will identify special problem instances that can be solved efficiently. 

In order to prove $\#$\textbf{P}-hardness, we focus on the following subclass of optimal transport problems with independent marginals, where~$\mu$ is the uniform distribution on~$\{0,1\}^K$, and~$\nu$ has only two support points.
\begin{center}
\fbox{\parbox{0.95\columnwidth}{ {\vspace{0.5ex}\centering \textsc{$\#$Optimal Transport} {\color{black} (for $p \geq 1$ fixed)} \\[1ex]}
	\textbf{Instance.} 
	Two support points $\bm y_{1},\bm y_2\in\mathbb R^K$, $\bm y_1\neq \bm y_2$, and a probability $t\in[0,1]$.
	\\[1ex]
    \textbf{Goal.} For $\mu$ denoting the uniform distribution on $\{0, 1\}^K$ and $\nu = t \delta_{\bm y_1} + (1-t)\delta_{\bm y_2}$, compute an approximation~$\widetilde W_c(\mu, \nu)$ of $W_c(\mu,\nu)$ for $c(\bm x, \bm y) = \|\bm x-\bm y\|^p$ such that the following hold.
    \begin{itemize}
        \item[(i)] The bit length of~$\widetilde W_c(\mu, \nu)$ is polynomially bounded in the bit length of the input~$(\bm y_1,\bm y_2,t)$.
        \item[(ii)] We have~$|\widetilde W_c(\mu, \nu)-W_c(\mu, \nu)|\leq\overline \varepsilon$, where
        \begin{equation*}
	    \overline{\varepsilon} = \frac{1}{4I}\, \min \Big\{ \left|\| \bm x_{i} -\bm y_1 \|^p - \|\bm x_{i}-\bm y_2 \|^p\right| : i \in \mathcal I, ~ \| \bm x_{i} -\bm y_1 \|^p - \left\|\bm x_{i}-\bm y_2 \right\|^p \neq 0 \Big\}
	\end{equation*}
	with $I = 2^K$ and~$\bm x_i$, $i \in \mathcal I$, representing the different binary vectors in $\{0, 1\}^K$.
    \end{itemize} 
    \vspace{0.5ex}
}}
\end{center}
We first need to show that the \textsc{$\#$Optimal Transport} problem is well-posed, that is, we need to ascertain the existence of a sufficiently accurate approximation that can be encoded in a polynomial number of bits. To this end, we first prove that the maximal tolerable approximation error $\overline\varepsilon$ is not too small.
\begin{lemma}
\label{lemma:polynomial_epsilon}
There exists~$\varepsilon \in (0, \overline{\varepsilon}]$ whose bit length is polynomially bounded in the bit length of~$(\bm y_1,\bm y_2,t)$.
\end{lemma}
\begin{proof} 
Note first that encoding an instance of the~\textsc{$\#$Optimal Transport} problem requires at least~$K$ bits because the~$K$ coordinates of~$\bm y_1$ and~$\bm y_2$ need to be enumerated. Note also that, by the definition of~$\overline\varepsilon$, there exists an index~$i^\star\in\mathcal I$ with~$\overline\varepsilon=\frac{1}{4I} | \| \bm x_{i^\star} -\bm y_1 \|^p - \|\bm x_{i^\star}-\bm y_2 \|^p|$. 
% We prove the claim first under the assumption that~$p$ is even, that is~$p=2q$ for some~$q\in\mathbb N$. In this case, we have
% \[
%     \overline\varepsilon=\frac{1}{4I}\left| \left( (\bm x_{i^\star} -\bm y_1 )^\top (\bm x_{i^\star} -\bm y_1 )\right)^q - \left( (\bm x_{i^\star} -\bm y_1 )^\top (\bm x_{i^\star} -\bm y_2 )\right)^q\right|.
% \]
{\color{black} As~$p\in [1,\infty)$, $\| \bm x_{i^\star}- \bm y_1\|^p$ and~$\| \bm x_{i^\star}- \bm y_2\|^p$ may be irrational numbers that cannot be encoded with any finite number of bits even if the vectors~$\bm y_1$ and~$\bm y_2$ have only rational entries. Thus, $\overline\varepsilon$ is generically irrational, in which case we need to construct~$\varepsilon\in (0,\overline{\varepsilon})$. To simplify notation, we henceforth use the shorthands~$a = \| \bm x_{i^\star} - \bm y_1 \|^2$ and~$b = \|\bm x_{i^\star} - \bm y_2\|^2$, which can be computed in polynomial time using~$\mathcal O(K)$ additions and multiplications. Without loss of generality, we may assume throughout the rest of the proof that~$a\geq b$. If~$a \geq b\geq 1$, then we have
\begin{align*}
\overline{\varepsilon} \!=\! \frac{1}{2^{K+2}} \left| a^{p/2} - b^{p/2}\right|  = \frac{1}{2^{K+2}} \left| \frac{a^p - b^p}{a^{p/2} + b^{p/2}}\right| 
&\geq  \frac{1}{2^{K+2}}\left|\frac{a^{\lfloor p \rfloor} - b^{\lfloor p \rfloor}}{a^{\lceil p/2 \rceil} + b^{\lceil p/2 \rceil}}\right| \triangleq \varepsilon > 0.
\end{align*}
Here, the first (weak) inequality holds because $a^{p - \lfloor p \rfloor} \geq 1$ and $( b / a )^{p - \lfloor p \rfloor} \leq 1$, which guarantees that 
\[ 
    \left| a^p - b^p \right| = a^{p - \lfloor p \rfloor} \left| a^{\lfloor p \rfloor} - \left( b / a \right)^{p - \lfloor p \rfloor} b^{\lfloor p \rfloor} \right| > \left| a^{\lfloor p \rfloor} - b^{\lfloor p \rfloor} \right|,
\]
whereas the second (strict) inequality follows from the construction of~$\overline\varepsilon$ as a strictly positive number, which implies that~$a \neq b$. The tolerance~$\varepsilon$ constructed in this way can be computed via~$\mathcal O(K)$ additions and multiplications, and as $p$ is not part of the input, its bit length is thus polynomially bounded. 
%Similarly, the evaluation of the denominator~$2^{K+2}$ requires~$\mathcal O(K)$ multiplications. Overall, $\varepsilon$ can therefore be computed in polynomial time, which trivially implies that the bit length of~$\varepsilon$ is polynomially bounded.
If~$a\geq 1 \geq b$ or $a,b\leq 1$, then~$\varepsilon$ can be constructed in a similar manner. Details are omitted for brevity.
}
% Assume now that~$p$ is odd, that is, $p=2q-1$ for some~$q\in\mathbb N$. In this case~$\| \bm x_{i^\star}- \bm y_1\|^p$ and~$\| \bm x_{i^\star}- \bm y_2\|^p$ may be irrational numbers that cannot be encoded with any finite number of bits even if the vectors~$\bm y_1$ and~$\bm y_2$ have only rational entries. Thus, $\overline\varepsilon$ may also be irrational, in which case we need to construct~$\varepsilon<\overline{\varepsilon}$. To simplify notation, we henceforth use the shorthands~$a = \| \bm x_{i^\star} - \bm y_1 \|^2$ and~$b = \|\bm x_{i^\star} - \bm y_2\|^2$, which can be computed in polynomial time using~$\mathcal O(K)$ additions and multiplications. If~$a,b\geq 1$, then we have
% %In the following, we will show that there exists an~$\epsilon>0$ that is polynomial in the input size of the problem, satisfying $a^p - b^p \geq \epsilon$. Now assume that $a \leq 1$ and let~$\varepsilon =({a^{2p} - b^{2p}})/({ a^{p-1} + b^{p-1}}) $. Then, we have
% \[
%     \overline\varepsilon = \frac{1}{2^{K+2}} \left|a^{p/2} - b^{p/2} \right| = \frac{1}{2^{K+2}}\left| \frac{a^{p} - b^{p}}{a^{p/2} + b^{p/2}} \right|> \frac{1}{2^{K+2}}\left| \frac{a^{p} - b^{p}}{ a^{q} + b^{q}}\right|\triangleq \varepsilon>0,
% \]
% where the first inequality holds because~$p/2<q$. The tolerance~$\varepsilon$ constructed in this way can be computed via~$\mathcal O(K)$ additions and multiplications, and therefore its bitlengh is polynomially bounded. If~$a\geq 1\geq b$, $a\leq 1\leq b$ or~$a,b\leq 1$, then~$\varepsilon$ can be constructed in a similar manner. Details are omitted for brevity.
\end{proof}

% {\color{black}[
% Extension to the case when $p$ is a rational number. ]
% \begin{proof}
% Recalling that $p$ is a rational number, we may assume $p = A + n/d$, where $A, n, d\in \mathbb Z_{+}$ and $d > 0$.
% As $p > 0$, we have $A \geq 0$ and $d > n \geq 0$. If $a\geq  b \geq 1$, then we have
% \begin{align*}
% \overline{\varepsilon} \!=\! \frac{1}{2^{K+2}} \left| a^{p/2} - b^{p/2}\right|  = \frac{1}{2^{K+2}} \left| \frac{a^p - b^p}{a^{p/2} + b^{p/2}}\right| 
% &> \frac{1}{2^{K+2}}\left|\frac{a^{A + \frac{n}{d}} - b^{A + \frac{n}{d}}}{a^{A\cdot d + n} + b^{A \cdot d + n}}\right|\\
% &\geq \frac{b^\frac{n}{d}}{2^{K+2}}.\left |\frac{a^{A} - b^A}{a^{A\cdot d + n} + b^{A \cdot d + n}}\right| \geq\frac{1}{2^{K+2}}\left|\frac{a^A - b^A}{a^{A \cdot d + n} + b^{A \cdot d + n}} \right| \triangleq \varepsilon > 0 ,
% \end{align*}
% where the first inequality follows because $d > 1/2$, the second inequality is due to $a \geq b$ and the third inequality follows because $a, b \geq 1$ and $d > n$. If~$a\geq 1\geq b$, $a\leq 1\leq b$ or~$a,b\leq 1$, then~$\varepsilon$ can be constructed in a similar manner. Details are omitted for brevity.

% \end{proof}
% \[|a^p - b^p| = |a^{A+ \frac{n}{d}} - b^{A + \frac{n}{d}}| \geq \min\{a^{\frac{n}{b}}, b^{\frac{n}{b}}\} \cdot |a^A - b^A| \geq \begin{cases}
% |a^A - b^A| & \text{if} \quad \min\{a^\frac{n}{b}, b^\frac{n}{b}\} \geq 1\\
% \min\{a, b\} \cdot |a^A - b^A|  &\text{o.w.}
% \end{cases}\]}
Lemma~\ref{lemma:polynomial_epsilon} readily implies that for any instance of the \textsc{$\#$Optimal Transport} problem there exists an approximate optimal transport distance~$\widetilde W_c(\mu, \nu)$ that satisfies both conditions~(i) as well as~(ii). For example, we could construct~$\widetilde W_c(\mu, \nu)$ by rounding the exact optimal transport distance~$W_c(\mu, \nu)$ to the nearest multiple of~$\varepsilon$. By construction, this approximation differs from~$W_c(\mu, \nu)$ at most by~$\varepsilon$, which is itself not larger than~$\overline\varepsilon$. In addition, this approximation trivially inherits the polynomial bit length from~$\varepsilon$. We emphasize that, in general, $\widetilde W_c(\mu, \nu)$ cannot be set to the exact optimal transport distance~$W_c(\mu, \nu)$, because~$W_c(\mu, \nu)$ may be irrational and thus have infinite bit length. However, Corollary~\ref{cor:bitlength-of-W} below implies that if~$p$ is even, then~$\widetilde W_c(\mu, \nu)=W_c(\mu, \nu)$ satisfies both conditions~(i) as well as~(ii).

Note that the \textsc{$\#$Optimal Transport} problem is parametrized by~$p$. The following example shows that if~$p$ was treated as an input parameter, then the problem would have exponential time complexity. 

\begin{example}
Consider an instance of the \textsc{$\#$Optimal Transport} problem with~$K=1$, $y_1=1$, $y_2=2$ and~$t=0$. In this case we have~$\mu=\frac{1}{2}\delta_0+\frac{1}{2}\delta_1$, $\nu=\delta_2$ and~$\overline\varepsilon=\frac{1}{8}$. An elementary analytical calculation reveals that~$W_c(\mu,\nu)=\frac{1}{2}(1+2^p)$. The bit length of any $\overline\varepsilon$-approximation~$\widetilde W_c(\mu,\nu)$ of~$W_c(\mu,\nu)$ is therefore bounded below by~$\log_2(\frac{1}{2}(1+2^p)-\frac{1}{8})\geq p-1$, which grows exponentially with the bit length~$\log_2(p)$ of~$p$. Note that the time needed for computing~$\widetilde W_c(\mu,\nu)$ is at least as large as its own bit length irrespective of the algorithm that is used. If~$p$ was an input parameter of the \textsc{$\#$Optimal Transport} problem, the problem's worst-case time complexity would therefore grow at least exponentially with its input size.
\end{example}

% We show later in Lemma~\ref{lemma:analytic}, if we restrict~$p$ to be 2, $\mu$ to be the uniform distribution on~$\{0, 1\}^K$ and $\nu$ to be any two point distribution, $W_c(\mu, \nu)$ admits an analytical solution that can be encoded in polynomially many bits. 

%Any problem instance with accuracy parameter~$\epsilon=0$ can be viewed as an instance of the classical optimal transport problem
%{\color{purple} Note that an instance of the \textsc{$\#$Optimal Transport} problem with $\epsilon=0$ is also an instance of the discrete optimal transport problems that can be encoded in time~$\text{poly}(K, J)$ because the product distribution~$\mu$ can be encoded enumerating all atoms and the corresponding probabilities. Hence, any complexity class assignment for an instance of the discrete optimal transport with independent marginals will be immediately true for any instance of a problem class that encapsulates it.}
The following main theorem shows that the \textsc{$\#$Optimal Transport} problem is hard even if~$p=2$.

{ \begin{theorem}[Hardness of \textsc{$\#$Optimal Transport}]
\label{theorem:approx-hard}
	\textsc{$\#$Optimal Transport} is $\#$\textbf{P}-hard for {\color{black}any $p \geq 1$.}
	%, $\mu$ is the uniform distribution on~$\{0,1\}^K$, and $\nu$ has only two atoms.
\end{theorem}}
We prove Theorem~\ref{theorem:approx-hard} by reducing the \textsc{$\#$Knapsack} problem to the \textsc{$\#$Optimal Transport} problem via a polynomial-time Turing reduction. To this end, we fix an instance of the \textsc{$\#$Knapsack} problem with input~$\bm w \in \mbb Z^K_{+}$ and $b \in \mbb Z_+$, and we denote by ${\nu}_ t  = t \delta_{\bm y_1} + (1-t) \delta_{\bm y_2}$ the two-point distribution with support points~$\bm y_1 = \bm 0$ and $\bm y_2=2b\bm{w}/ \|\bm w\|^2 $, whose probabilities are parameterized by $t \in [0, 1]$. Recall also that $\mu$ is the uniform distribution on $\{0,1\}^K$, that is, $\mu = \frac{1}{I} \sum_{i \in \mathcal I} \delta_{\bm x_i}$, where $I = 2^K$ and $\{\bm x_i: i \in \mathcal I \} = \{0, 1\}^K$. Without loss of generality, we may assume that the support points of $\mu$ are ordered so as to satisfy
\begin{align*}
    \| \bm x_1 - \bm y_1 \|^p - \| \bm x_1 - \bm y_2 \|^p \leq \| \bm x_2 - \bm y_1 \|^p- \| \bm x_2 - \bm y_2 \|^p \leq \cdots \leq \| \bm x_I - \bm y_1 \|^p - \| \bm x_I -  \bm y_2 \|^p.
\end{align*}
Below we will demonstrate that computing~$W_c(\mu, \nu_t)$ approximately is at least as hard as solving the \textsc{$\#$Knapsack} problem, which amounts to evaluating the cardinality of $\mathcal I(\bm w, b) = \{\bm x \in \{0,1\}^K :  \bm  w^\top \bm x \leq b\}$. 

%To this end, select any instance of the \textsc{$\#$Knapsack} problem with inputs $\bm w \in \mbb Z^K_{+}$ and $b \in \mbb Z_+$

\begin{lemma}
    \label{lemma:analytic}
    If $c(\bm x, \bm y)=\|\bm x- \bm y \|^p$ for some $p\ge 1$, then the optimal transport distance $W_c(\mu, \nu_t)$ is continuous, piecewise affine and convex in $t\in[0,1]$. Moreover, it admits the closed-form formula
    \begin{align}
    \label{eq:analytic}
    \begin{array}{l}
         \DS W_c(\mu, \nu_t) =  \frac{1}{I} \sum_{i=1}^{\lfloor t I \rfloor} \| \bm x_i -\bm y_1 \|^p + \frac{1}{I} \sum_{i=\lfloor t I \rfloor + 1}^I \left\|\bm x_i-\bm y_2 \right\|^p \\[3ex]
         \DS \hspace{12.9em} + \frac{(tI - \lfloor t I \rfloor)}{I}\left( \| \bm x_{\lfloor t I \rfloor + 1} -\bm y_1 \|^p - \|\bm x_{\lfloor t I \rfloor + 1}-\bm y_2 \|^p \right).
    \end{array}
	\end{align}
\end{lemma}
\begin{proof}
    For any fixed~$t \in [0, 1]$, the discrete optimal transport problem~\eqref{eq:primal} satisfies
	\begin{align*}
	W_c(\mu,  {\nu}_ t)
	&= \min\limits_{\bm \pi \in \Pi(\mu,\nu_t)} ~ \sum_{i \in \mathcal I} \sum_{j \in \mathcal J} \| \bm x_i - \bm y_j \|^p \pi_{ij} \\[0.5ex]
	&= \left\{
	\begin{array}{cl} \min\limits_{\bm q_{1}, \bm q_{2} \in \R_+^I} & \DS t \sum_{i \in \mathcal I} \| \bm x_i - \bm y_1 \|^p q_{1,i} + (1-t) \sum_{i \in \mathcal I} \|\bm x_i - \bm y_2 \|^p q_{2,i} \\ [3ex]
	\textrm{s.t.}& t \bm q_1 +(1-t) \bm q_2 = \bm 1 / I, ~ \bm 1^\top \bm q_1 = 1, ~ \bm 1^\top \bm q_2 = 1.
	\end{array}\right.
	\end{align*}
	The second equality holds because the transportation plan can be expressed as
	\[
	    \pi = \sum_{i \in \mathcal I} \sum_{j \in \mathcal J} \pi_{ij} \delta_{(\bm x_i, \bm y_j)} = t \cdot q_1 \otimes \delta_{\bm y_1} + (1-t) \cdot q_2 \otimes \delta_{\bm y_2},
	\]
	with $q_j = \sum_{i \in \mathcal I} q_{j,i} \delta_{\bm x_i}$ representing the conditional distribution of~$\bm x$ given~$\bm y=\bm y_j$ under~$\pi$ for every~$j = 1, 2$. This is a direct consequence of the law of total probability. By applying the variable transformations $\bm q_1 \leftarrow t I \bm q_1$ and $\bm q_2 \leftarrow (1-t) I \bm q_2$ to eliminate all bilinear terms, we then find
	\begin{align}
    \label{eq:Wc:with:t}
    W_c(\mu, \nu_t)=\left\{\begin{array}{cll}
	\DS \min_{\bm q_1, \bm q_2 \in \R_+^I} & \DS \frac{1}{I} \sum_{i \in \mathcal I} \| \bm x_i -\bm y_1 \|^p q_{1,i} + \frac{1}{I} \sum_{i \in \mathcal I} \left\|\bm x_i-\bm y_2 \right\|^p  q_{2,i} \\[3ex]
	\textrm{s.t.} & \DS \bm 1^\top \bm q_{1} =  t I, ~ \bm 1^\top \bm q_{2} = (1 - t) I, ~ \bm q_1 +  \bm q_2 = \bm 1.
	\end{array}\right.
    \end{align}
    Observe that~\eqref{eq:Wc:with:t} can be viewed as a parametric linear program. By~\cite[Theorem~6.6]{dantzig2003linear}, its optimal value~$W_c(\mu, \nu_t)$ thus constitutes a continuous, piecewise affine and convex function of~$t$. It remains to be shown that~$W_c(\mu, \nu_t)$ admits the analytical expression~\eqref{eq:analytic}. To this end, note that the decision variable $\bm q_2$ and the constraint~$\bm q_1 +  \bm q_2 = \bm 1$ in problem~\eqref{eq:Wc:with:t} can be eliminated by applying the substitution $\bm q_2 \gets \bm 1 - \bm q_1 $. Renaming $\bm q_1$ as $\bm q$ to reduce clutter, problem~\eqref{eq:Wc:with:t} then simplifies to
    \begin{align}
    \label{eq:Wc:without:q2}
    \begin{array}{cl}
    \DS \min_{\bm q \in \mathbb R^I} & \DS \frac{1}{I} \sum_{i \in \mathcal I} \left( \| \bm x_i -\bm y_1 \|^p - \left\|\bm x_i-\bm y_2 \right\|^p \right) q_{i} + \frac{1}{I} \sum_{i \in \mathcal I} \left\|\bm x_i-\bm y_2 \right\|^p \\[3ex]
    \text{s.t.} & \DS \bm 1^\top \bm q = t I, ~ \bm 0 \leq \bm q \leq \bm 1.
    \end{array}
    \end{align}
    Recalling that the atoms of $\mu$ are ordered such that $\| \bm x_1 - \bm y_1 \|^p - \| \bm x_1 - \bm y_2 \|^p \leq \dots \leq \| \bm x_I - \bm y_1 \|^p - \| \bm x_I - \bm y_2 \|^p$, one readily verifies that problem~\eqref{eq:Wc:without:q2} is solved analytically by
    \begin{align*}
	q_{i}^\star = \left\{\begin{array}{ll}
	1 & \text{if } i \leq \lfloor t I \rfloor \\
	t I - \lfloor t I \rfloor  & \text{if } i = \lfloor t I \rfloor + 1 \\
	0 & \text{if } i > \lfloor t I \rfloor + 1.
	\end{array} \right. 
	\end{align*}
	Substituting $\bm q^\star$ into~\eqref{eq:Wc:without:q2} yields~\eqref{eq:analytic}, and thus the claim follows.
\end{proof}

Lemma~\ref{lemma:analytic} immediately implies that the bit length of~$W_c(\mu, \nu_t)$ is polynomially bounded.

\begin{corollary}
\label{cor:bitlength-of-W}
If $c(\bm x, \bm y) = \| \bm x - \bm y \|^p$ and~$p$ is even, then the bit length of the optimal transport distance~$W_c(\mu, \nu_t)$ grows at most polynomially with the bit length of~$(\bm y_1, \bm y_2, t)$.
\end{corollary}
\begin{proof}
The bit length of~$(\bm y_1, \bm y_2, t)$ is finite if and only if all of its components are rational and thus representable as ratios of two integers. We denote by~$U\in\mathbb N$ the maximum absolute value of these integers. 

For ease of exposition, we assume first that~$p=2$ and~$t=1$. In addition, we use~$D\in\mathbb N$ to denote the least common multiple of the denominators of the~$K$ components of~$\bm y_1$. It is easy to see that~$D\leq U^{K}$. By Lemma~\ref{lemma:analytic}, the optimal transport distance~$W_c(\mu, \nu_t)$ can thus be expressed as the average of the~$I$ quadratic terms~$\|\bm x_i - \bm y_1 \|^2=\bm x_i^\top \bm x_i+2\bm x_i^\top \bm y_1 +\bm y_1^\top\bm y_1$ for~$i\in\mc I$. Each such term is equivalent to a rational number with denominator~$D^2$ and a numerator that is bounded above by~$K(1+2U+U^2)D^2$. Indeed, each component of~$\bm x_i$ is binary, whereas each component of~$\bm y_1$ can be expressed as a rational number with denominator~$D$ and a numerator with absolute value at most~$UD$. By Lemma~\ref{lemma:analytic}, $W_c(\mu, \nu_t)$ is thus representable as a rational number with denominator~$ID^2$ and a numerator with absolute value at most~$IK(1+U)^2D^2$. Therefore, the number of bits needed to encode~$W_c(\mu, \nu_t)$ is at most of the order
\[
    \mc O\left(\log_2(IKU^2D^2))\right)\leq \mc O\left(\log_2(2^KKU^2U^{2K})\right) = \mc O\left(K\log_2(U)\right),
\]
where the inequality holds because~$I=2^K$ and~$D\leq U^K$. As both~$K$ and~$\log_2(U)$ represent lower bounds on the bit length of~$(\bm y_1, \bm y_2, t)$, we have thus shown that the bit length of~$W_c(\mu, \nu_t)$ is indeed polynomially bounded in the bit length of~$(\bm y_1, \bm y_2, t)$. If~$p$ is any even number and~$t$ any rational probability, then the claim can be proved using similar---yet more tedious---arguments. Details are omitted for brevity.
\end{proof}

Corollary~\ref{cor:bitlength-of-W} implies that the optimal transport distance~$W_c(\mu, \nu_t)$ is rational whenever~$p$ is an even integer and~$t$ is rational. Otherwise, $W_c(\mu, \nu_t)$ is generically irrational because the Euclidean norm of a vector $\bm v=(v_1,\ldots,v_K)$ is irrational unless $(v_1,\ldots, v_K, \|\bm v\|)$ is proportional to a Pythagorean $(K+1)$-tuple, where the inverse proportionality factor is itself equal to the square of an integer. We will now show that the cardinality of the set~$\mathcal I(\bm w, b)$ can be computed by solving the univariate minimization problem
\begin{align}
    \label{eq:min:Wc}
    \min_{t \in [0, 1]} W_c(\mu, \nu_t).
\end{align}
% This insight is formalized in the next lemma.
\begin{lemma}
	\label{lemma:knapsack}
	If $c(\bm x, \bm y)=\|\bm x- \bm y \|^p$ for some $p\ge 1$, then $t^\star = |\mathcal I(\bm w, b)|/I$ is an optimal solution of problem~\eqref{eq:min:Wc}. If in addition each component of~$\bm w$ is even and~$b$ is odd, then~$t^\star$ is unique. 
\end{lemma}
\begin{proof}
	From the proof of Lemma~\ref{lemma:analytic} we know that the optimal transport distance~$W_c(\mu, \nu_t)$ coincides with the optimal value of~\eqref{eq:Wc:with:t}. Thus, problem~\eqref{eq:min:Wc} can be reformulated as
    \begin{equation}
	\label{eq:min:Wc:with:t}
	\begin{array}{cll}
	\DS \min_{\substack{t \in [0,1] \\ \bm q_1, \bm q_2 \in \R_+^I}} & \DS \frac{1}{I} \sum_{i \in \mathcal I} \| \bm x_i -\bm y_1 \|^p q_{1,i} + \frac{1}{I} \sum_{i \in \mathcal I} \left\|\bm x_i-\bm y_2 \right\|^p  q_{2,i} \\[3ex]
	\textrm{s.t.} & \DS \bm 1^\top \bm q_{1} =  t I, ~ \bm 1^\top \bm q_{2} = (1 - t) I, ~ \bm q_1 +  \bm q_2 = \bm 1.
	\end{array}
	\end{equation}
	Note that the decision variable $t$ as well as the two normalization constraints for~$\bm q_1$ and~$\bm q_2$ are redundant and can thus be removed without affecting the optimal value of~\eqref{eq:min:Wc:with:t}. In other words, there always exists $t\in[0,1]$ such that $\bm 1^\top \bm q_1 =  t I$ and $\bm 1^\top \bm q_2 = (1 - t) I$.
	Hence, \eqref{eq:min:Wc:with:t} simplifies to
	\begin{align}
	\label{eq:min:Wc:without:t}
	\begin{array}{cll}
	\DS \min_{\bm q_1, \bm q_2 \in \R_+^I} & \DS \frac{1}{I} \sum_{i \in \mathcal I} \| \bm x_i -\bm y_1 \|^p q_{1,i} + \frac{1}{I} \sum_{i \in \mathcal I} \left\|\bm x_i-\bm y_2 \right\|^p  q_{2,i} \\[3ex]
	\textrm{s.t.} & \bm q_1 +  \bm q_2 = \bm 1.
	\end{array}
	\end{align} 
	
	Next, introduce the disjoint index sets
	\begin{align*}
	    \mathcal I_0 &= \{i \in \mc I : \|\bm x_i - \bm y_1\| = \|\bm x_i - \bm y_2\|\}\\
	    \mathcal I_{1} &= \{ i \in \mathcal I: \|\bm x_i - \bm y_1 \| < \|\bm x_i-\bm y_{2}\| \} \\
	    \mathcal I_{2} &= \{ i \in \mathcal I: \|\bm x_i - \bm y_1 \| > \|\bm x_i-\bm y_{2}\| \},
	\end{align*}
	which form a partition of $\mathcal I$. Using these sets, optimal solution of problem~\eqref{eq:min:Wc:without:t} can be expressed as
	\begin{align}
	\label{eq:solution}
	q_{1,i}^\star = \left\{\begin{array}{ll}
	\theta_i & \text{if } i \in \mc I_0\\
	1 & \text{if } i \in \mathcal I_1 \\
	0 & \text{if } i \in \mathcal I_2
	\end{array} \right. 
	\quad \text{and} \quad
	q_{2,i}^\star = \left\{\begin{array}{ll}
	1- \theta_i & \text{if } i \in \mc I_0\\
	0 & \text{if } i \in \mathcal I_1 \\
	1 & \text{if } i \in \mathcal I_2
	\end{array} \right. 
	\end{align}
    Therefore, we have
	\begin{align*}
	\min_{t \in [0,1]} W_c(\mu, {\nu}_ t )
	= \frac{1}{I} \sum_{i \in \mathcal I} \min \Big\{\|\bm x_i -\bm y_1 \|^p,\left\|\bm x_i - \bm y_2 \right\|^p \! \Big\}.
	\end{align*}
    Any minimizer~$(\bm q^\star_1,\bm q^\star_2)$ of~\eqref{eq:min:Wc:without:t} gives thus rise to a minimizer~$(t^\star, \bm q^\star_1,\bm q^\star_2)$ of~\eqref{eq:min:Wc:with:t}, where $t^\star = (\bm 1^\top \bm q_1^\star) / I$. 
	Moreover, the minimizers of~\eqref{eq:min:Wc} are exactly all numbers of the form $t^\star = (\bm 1^\top \bm q_1^\star) / I$ corresponding to the minimizer~$(\bm q^\star_1,\bm q^\star_2)$ of~\eqref{eq:min:Wc:without:t}. In view of~\eqref{eq:solution}, this observation allows us to conclude that
	\begin{equation}
	    \label{eq:argmin_wass_interval}
	    \argmin_{t \in [0, 1]} W_c(\mu, \nu_t) =  \left[{|\mathcal I_1|}/{I}, |\mc I_0 \cup \mc I_1|/ I\right].
	\end{equation}
	By the definitions of $\mathcal I(\bm w, b)$, $\bm y_1$ and $\bm y_2$, it is further evident that
	\begin{align*}
	|\mathcal I (\bm w, b)| 
	=& \left| \left\{i \in \mathcal I : \bm w^\top \bm x_i \leq b \right\} \right| 
	= \left| \left\{ i \in \mathcal I: \|\bm x_i - \bm y_1 \|^2 \leq \left\|\bm x_i - \bm y_2 \right\|^2 \right\}\right|
	= |\mc I_1 \cup \mc I_0 |.
	\end{align*}
	Therefore, we may finally conclude that 
	$$ |\mathcal I (\bm w, b)| / I \in \argmin_{t \in [0, 1]} W_c(\mu, \nu_t).$$
	
	Assume now that each component of~$\bm w$ is even and~$b$ is odd. In this case, there exists no~$\bm x\in \{0,1\}^K$ that satisfies~$\bm x^\top \bm w = b$ and consequentially~$\mc I_0$ is empty. Consequently, the interval of minimizers in~\eqref{eq:argmin_wass_interval} collapses to the singleton~$|\mc I_1| / I = |\mc I(\bm w, b)|/I$.
	This observation completes the proof.
\end{proof}
Armed with Lemmas~\ref{lemma:analytic} and \ref{lemma:knapsack}, we are now ready to prove Theorem~\ref{theorem:approx-hard}.

\begin{proof}[Proof of~Theorem~\ref{theorem:approx-hard}]
    Select an instance of the \textsc{$\#$Knapsack} problem with input~$\bm w \in \mbb Z^K_{+}$ and $b \in \mbb Z_+$. Throughout this proof we will assume without loss of generality that each component of~$\bm w$ is even and that~$b$ is odd. Indeed, if this was not the case, we could replace~$\bm w$ with~$\bm w'=2\bm w$ and~$b$ with~$b'=2b+1$. It is easy to verify that the two instances of the \textsc{$\#$Knapsack} problem with inputs~$(\bm w,b)$ and~$(\bm w',b')$ have the same solution. In addition, the bit length of~$(\bm w',b')$ is polynomially bounded in the bit length of~$(\bm w,b)$.
    
    Given~$\bm w$ and~$b$, define the distributions~$\mu$ and ${\nu}_t$ for~$t\in[0,1]$ as well as the set $\mathcal I(\bm w, b)$ in the usual way.
    From Lemma~\ref{lemma:analytic} we know that~$W_c(\mu, \nu_t)$ is continuous, piecewise affine and convex in $t$. The analytical formula~\eqref{eq:analytic} further implies that~$W_c(\mu, \nu_t)$ is affine on the interval~$[(i-1)/I,i/I]$ with slope~{\color{black} $a_i \cdot I$}, where
    \begin{align}
        \label{eq:slopes}
        a_i = W_c (\mu, \nu_{i/I}) - W_c (\mu, \nu_{(i-1)/I}) \qquad \forall i\in\mathcal I.
    \end{align}
    Thus, \eqref{eq:min:Wc} constitutes a univariate convex optimization problem with a continuous piecewise affine objective function.
    As each component of~$\bm w$ is even and~$b$ is odd, Lemma~\ref{lemma:knapsack} implies that $t^\star=|\mathcal I(\bm w, b)| / I$ is the unique minimizer of~\eqref{eq:min:Wc}. Therefore, the given instance of the \textsc{$\#$Knapsack} problem can be solved by solving~\eqref{eq:min:Wc} and multiplying its unique minimizer~$t^\star$ with~$I$.
    
    In the following we will first show that if we had access to an oracle that computes~$W_c(\mu, \nu_t)$ exactly, then we could construct an algorithm that finds~$t^\star$ and the solution~$t^\star I$ of the \textsc{$\#$Knapsack} problem by calling the oracle~$2K$ times (Step~1). Next, we will prove that if we had access to an oracle that solves the \textsc{$\#$Optimal Transport} problem and thus outputs only approximations of~$W_c(\mu, \nu_t)$, then we could extend the algorithm from Step~1 to a polynomial-time Turing reduction from the \textsc{$\#$Knapsack} problem to the \textsc{$\#$Optimal Transport} problem (Step~2). Step~2 implies that \textsc{$\#$Optimal Transport} is $\#$\textbf{P}-hard.
    
    {\em Step~1.} Assume now that we have access to an oracle that computes~$W_c(\mu, \nu_t)$ exactly. In addition, introduce an array $\bm a = (a_0, a_1, \dots, a_I)$ with entries~$a_i$, $i\in\mathcal I$, defined as in~\eqref{eq:slopes} and with~$a_0=-\infty$. 
    % \begin{align*}
    % a_n = W_c (\mu, \nu_{(n+1)/I}) - W_c (\mu, \nu_{n/I}) \qquad \forall n < I
    % \end{align*}
    % and $a_I = +\infty$. 
    Thus, each element of~$\bm a$ can be evaluated with at most two oracle calls. %Next, fix any~$n<I$ and note that $W_c(\mu, \nu_t)$ is affine on the interval~$[n/I,(n+1)/I]$ thanks to the analytical formula~\eqref{eq:analytic}. By the definition of the right derivative~$g(t)$, one therefore readily verifies that~$a_n = g(n/I) / I$. In addition, as $W_c(\mu, \nu_t)$ is convex, its right derivative~$g(t)$ 
    The array~$\bm a$ is useful because it contains all the information that is needed to solve the univariate convex optimization problem~\eqref{eq:min:Wc}. Indeed, as~$W_c(\mu, \nu_t)$ is a convex piecewise linear function with slope~{\color{black} $a_i \cdot I$} on the interval~$[i/I,(i-1)/I]$, the array~$\bm a$ is sorted in ascending order, and the unique minimizer~$t^\star$ of~\eqref{eq:min:Wc} satisfies
    \begin{align}
        \label{eq:binary}
        |\mathcal I(\bm w, b)| = t^\star I = \max \big\{ i \in \mathcal I \cup \{ 0 \} : a_i \leq 0 \big\}.
    \end{align}
    In other words, counting all elements of the set~$\mathcal I(\bm w, b)$ and thereby solving the \textsc{$\#$Knapsack} problem  is equivalent to finding the maximum index $i\in\mathcal I\cup\{0\}$ that meets the condition $a_i\leq 0$. The binary search method detailed in Algorithm~\ref{algorithm:binary} efficiently finds this index. Binary search methods are also referred to as half-interval search or bisection algorithms, and they represent iterative methods for finding the largest number within a sorted array that is smaller or equal to a given threshold (0 in our case). Algorithm~\ref{algorithm:binary} first checks whether the number in the middle of the array is non-positive. Depending on the outcome, either the part of the array to the left or to the right of the middle element may be discarded because the array is sorted. %, and otherwise  and the search continues on the remaining half, again checking whether the middle element of the array meets the condition, and repeating this process 
    This procedure is repeated until the array collapses to the single element corresponding to the sought number. As the length of the array is halved in each iteration, the binary search method applied to an array of length~$I$ returns the solution in~$\log_2 I=K$ iterations  \cite[\S~12]{cormen2009introduction}.
    
    \begin{table}[H]
	\centering
	\begin{minipage}{0.8\textwidth}
    \vspace{-1em}
	\begin{algorithm}[H]		
	\caption{Binary search method \label{algorithm:binary}}
		\begin{algorithmic}[1]
			\Require An array $\bm a\in \mathbb R^I$ with~$I=2^K$ sorted in ascending order
			\State Initialize $\underline{n} = 0$ and $\overline{n} = I$
			\For{$k=1, \ldots, K$}
			\State \hspace{-1ex}Set $n \gets {(\overline{n} + \underline{n})}/{2}$
			\State \hspace{-1ex}\algorithmicif~~$a_n \leq 0$~~\algorithmicthen~~$\underline{n} \gets n$~~\algorithmicelse~~$\overline{n} \gets n$
			\EndFor	\vspace{0.1em}
			\State \algorithmicif~~$a_{\underline n} \leq 0$~~\algorithmicthen~~$n \gets \underline{n}$~~\algorithmicelse~~$n \gets \overline{n}$
			\Ensure ${n}$
		\end{algorithmic}
	\end{algorithm}
	\vspace{-1em}
	\end{minipage}
    \end{table}
    
    One can use induction to show that, in any iteration~$k$ of Algorithm~\ref{algorithm:binary}, $n$ is given by a multiple of~$2^{K-k}$ and represents indeed an eligible index. Similarly, in any iteration~$k$ we have~$\overline n-\underline n=2^{K-k+1}$.
    
    {\em Step~2.} Assume now that we have only access to an oracle that solves the \textsc{$\#$Optimal Transport} problem, which merely returns an approximation~$\widetilde W_c(\mu, \nu_t)$ of~$W_c (\mu, \nu_t)$. Setting~$\widetilde a_0=-\infty$ and
    \begin{align}
        \label{eq:approximate-slopes}
        \widetilde a_i =\widetilde W_c (\mu, \nu_{i/I}) - \widetilde W_c (\mu, \nu_{(i-1)/I}) \qquad \forall i\in\mathcal I,
    \end{align}
    we can then introduce a perturbed array~$\widetilde {\bm a}=(\widetilde a_0, \widetilde a_1,\ldots, \widetilde a_I)$ which provides an approximation for~$\bm a$. In the following we will prove that, even though~$\widetilde {\bm a}$ is no longer necessarily sorted in ascending order, the sign of~$\widetilde a_i$ coincides with the sign of~$a_i$ for every~$i\in\mathcal I$. Algorithm~\ref{algorithm:binary} therefore outputs the exact solution~$|\mathcal I(\bm w, b)|$ of the \textsc{$\#$Knapsack} problem even if its input~$\bm a$ is replaced with~$\widetilde {\bm a}$. To see this, we first note that
    \begin{align}
        \label{eq:ai-theoretical}
        a_i = \frac{1}{I} \left( \| \bm x_{i} - \bm y_1 \|^p - \| \bm x_{i} - \bm y_2 \|^p \right) \qquad \forall i\in\mathcal I,
    \end{align}
    which is an immediate consequence of the analytical formula~\eqref{eq:analytic} for~$W_c(\mu, \nu_t)$. We emphasize that~\eqref{eq:ai-theoretical} has only theoretical relevance but cannot be used to evaluate~$a_i$ in practice because it relies on our assumption that the support points~$\bm x_i$, $i\in\mathcal I$, are ordered such that~$\| \bm x_i - \bm y_1 \|^p - \| \bm x_i - \bm y_2 \|^p$ is non-decreasing in~$i$. Indeed, there is no efficient algorithm for ordering these~$2^K$ points in practice. Using~\eqref{eq:ai-theoretical}, we then find
    \begin{align*}
        \overline\varepsilon& =\frac{1}{4} \min_{i\in\mathcal I} \left\{ |a_i| : a_i\neq 0\right\}=\frac{1}{4} \min_{i\in\mathcal I} |a_i|,
    \end{align*}
    where the first equality follows from the definition of~$\overline\varepsilon$, and the second equality holds because each component of~$\bm w$ is even and~$b$ is odd, which implies that~$\| \bm x_i - \bm y_1 \| \neq \|\bm x_i - \bm y_2\|$ and thus~$a_i\neq 0$ for all~$i\in\mathcal I$. The last formula for~$\overline\varepsilon$ immediately implies that~$|a_i|\geq 4\overline\varepsilon$ for all~$i\in\mathcal I$. Together with the estimate
    \[
        |\widetilde a_i - a_i| \leq \Big| \widetilde W_c (\mu, \nu_{i/I}) -W_c (\mu, \nu_{i/I})\Big|+\Big| \widetilde W_c (\mu, \nu_{(i-1)/I}) - W_c (\mu, \nu_{(i-1)/I}) \Big|\leq {\color{black} 2\overline\varepsilon},
    \]
    this implies that~$\widetilde a_i$ has indeed the same sign as~$a_i$ for every~$i\in\mathcal I$. As the execution of Algorithm~\ref{algorithm:binary} depends on the input array only through the signs of its components, Algorithm~\ref{algorithm:binary} with input~$\widetilde{\bm a}$ computes indeed the exact solution~$|\mathcal I(\bm w, b)|$ of the \textsc{$\#$Knapsack} problem. If the perturbed slope~$\widetilde a_n$ in line~4 of Algorithm~\ref{algorithm:binary} is evaluated via~\eqref{eq:approximate-slopes} by calling the \textsc{$\#$Optimal Transport} oracle twice, then Algorithm~\ref{algorithm:binary} constitutes a Turing reduction from the $\#$\textbf{P}-hard \textsc{$\#$Knapsack} problem to the \textsc{$\#$Optimal Transport} problem.
    
    To prove that the \textsc{$\#$Optimal Transport} problem is $\#$\textbf{P}-hard, it remains to be shown that if any oracle call requires unit time, then the Turing reduction constructed above runs in polynomial time in the bit length of~$(\bm{w},b)$. This is indeed the case because Algorithm~\ref{algorithm:binary} calls the \textsc{$\#$Optimal Transport} oracle only $2K$ times in total and because all other operations can be carried out efficiently. In particular, the time needed for reading the oracle outputs is polynomially bounded in the size of~$(\bm w, b)$. Indeed, the bit length of~$\widetilde W_c(\mu,\nu_{i/I})$ is polynomially bounded in the bit length of~$(\bm y_1,\bm y_2,i/I)$ thanks to the definition of the \textsc{$\#$Optimal Transport} problem, and the time needed for computing~$(\bm y_1,\bm y_2,i/I)$ is trivially bounded by a polynomial in the bit length of~$(\bm{w},b)$ for any~$i\in\mathcal I$. These observations complete the proof.
   % In summary, we can evaluate the cardinality of the feasible set $\mathcal I(\bm w, b)$ of the $0/1$-knapsack problem exactly by executing $K$ iterations of the binary search Algorithm~\ref{algorithm:binary}, in each of which we need to call an imprecise {\color{black} optimal transport} oracle twice for some rational test points $t\in \{0, 1/ I, 2/I, \dots, 1\}$. Note that all involved operations except for the approximate computation of~$W_c(\mu, \nu_t)$ {\color{black} within an absolute accuracy~$\overline \varepsilon$} can be carried out in time polynomial in the bit length of $\bm{w}$ and $b$. Thus, if we could approximate~$W_c(\mu, \nu_t)$ in time polynomial in the bit length of $\bm{w}$, $b$ and $t$, then we could efficiently solve the \textsc{$\#$Knapsack} problem, which is known to be $\#$\textbf{P}-hard as discussed in Section~\ref{sec:knapsack}. We have thus constructed a polynomial-time Turing reduction from the \textsc{$\#$Knapsack} problem to the problem of computing the optimal transport distance~$W_c(\mu, \nu_ t)$ approximately within an absolute error of~$\overline\varepsilon$. By the definition of the complexity class~$\#$\textbf{P} (see, {\em e.g.}, \cite[Definition~1]{van1990handbook}), we may thus conclude that computing $W_c(\mu, \nu_t)$ approximately within an absolute error of~$\overline\varepsilon$ is $\#$\textbf{P}-hard. Thus, the claim follows.
\end{proof}

We emphasize that the Turing reduction derived in the proof of Theorem~\ref{theorem:approx-hard} can be implemented without knowing the accuracy level~$\overline\varepsilon$ of the \textsc{$\#$Optimal Transport} oracle. This is essential because~$\overline\varepsilon$ is defined as the minimum of exponentially many terms, and we are not aware of any method to compute it efficiently. Without such a method, a Turing reduction relying on~$\overline\varepsilon$ could not run in polynomial time.

\begin{remark}[Polynomial-Time Turing Reductions]
Recall that a polynomial-time Turing reduction from problem~$A$ to problem~$B$ is a Turing reduction that runs in polynomial time in the input size of~$A$ under the hypothetical assumption that there is an oracle for solving~$B$ in unit time. The time needed for computing oracle inputs and reading oracle outputs is attributed to the Turing reduction and is not absorbed in the oracle. Thus, a Turing reduction can run in polynomial time only if the oracle's output size is guaranteed to be polynomially bounded. The existence of a polynomial-time Turing reduction from~$A$ to~$B$ implies that if there was an efficient algorithm for solving~$B$, then we could solve~$A$ in polynomial time (this operationalizes the assertion that~``$A$ is not harder than~$B$''). One could use this implication as an alternative definition, that is, one could define a polynomial-time Turing reduction as a Turing reduction that runs in polyonomial time provided that the oracle runs in polynomial time. In our opinion, this alternative definition would be perfectly reasonable. However, it is not equivalent to the original definition by~\citet{valiant1979complexitypermanent}, which compels us to ascertain that the oracle output has polynomial size irrespective of the oracle's actual runtime. Instead, the alternative definition directly refers to the oracle's actual runtime. In that it conditions on oracles that run in polynomial time, it immediately guarantees that their outputs have polynomial size. In short, the original definition requires the bit length of the oracle's output to be polynonmially bounded for {\em every} oracle that solves~$B$ (which requires a proof), whereas the alternative definition requires such a bound only for oracles that solve~$B$ in polynomial time (which requires no proof). As Theorem~\ref{theorem:approx-hard} relies on the original definition of a polynomial-time Turing reduction, we had to introduce condition~(ii) in the definition of the \textsc{$\#$Optimal Transport} problem. We consider the differences between the original and alternative definitions of polynomial-time Turing reductions as pure technicalities, but discussing them here seems relevant for motivating our formulation of the \textsc{$\#$Optimal Transport} problem.
\end{remark}

Assume now that~$p$ is an even number, and consider any instance of the \textsc{$\#$Optimal Transport} problem. In this case, all coefficients of the linear program~\eqref{eq:primal} are rational, and thus~$W_c(\mu, \nu_t)$ is a rational number that can be computed in finite time ({\em e.g.}, via the simplex algorithm). From Corollary~\ref{cor:bitlength-of-W} we further know that~$W_c(\mu, \nu_t)$ has  polynomially bounded bit length. Thus, $\widetilde W_c(\mu, \nu_t)=W_c(\mu, \nu_t)$ satisfies both properties~(i) and~(ii) that are required of an admissible approximation of the optimal transport distance. Nevertheless, Theorem~\ref{theorem:approx-hard} asserts that computing~$W_c(\mu, \nu_t)$ approximately is already $\#$\textbf{P}-hard. This trivially implies that computing ~$W_c(\mu, \nu_t)$ {\em exactly} is also $\#$\textbf{P}-hard.

\section{Dynamic Programming-Type Solution Methods}
\label{sec:polynomial}
We now return to the generic optimal transport problem with independent marginals, where $\mu $ is representable as~$\otimes_{k \in \mathcal K} \mu_k$, the marginals of~$\mu$ constitute arbitrary univariate distributions supported on~$L$ points, and~$\nu$ constitutes an arbitrary multivariate distribution supported on~$J$ points. This problem class covers all instances of the \textsc{$\#$Optimal Transport} problem, and by Theorem~\ref{theorem:approx-hard} it is therefore $\#$\textbf{P}-hard even if only approximate solutions are sought. In fact, {\em any} problem class that is rich enough to contain all instances of the \textsc{$\#$Optimal Transport} problem is $\#$\textbf{P}-hard. %In this section we will focus on complexity of an~$\#$\textbf{P}-hard problem by investigating its subproblems and trying to map the boundary between those subproblems that are polynomially solvable and those that are~$\#$\textbf{P}-hard. In particular, 
We will now demonstrate that{\color{black}, for $p=2$,} particular instances of the optimal transport problem with independent marginals can be solved in polynomial or pseudo-polynomial time by a dynamic programming-type algorithm even though the distribution~$\mu$ involves exponentially many atoms and the linear program~\eqref{eq:primal} has exponential size. %One class of such instances is described below. We further show that investigation of a psedo-polynomial algorithm provides insights about the instances of \textsc{$\#$Optimal Transport} problem that are solvable in polynomial time. 
Throughout this discussion we call~$\mc N\subseteq \mathbb R$ a one-dimensional regular grid with cardinality~$N$ if there exist~$\hat s_{1}, \ldots, \hat s_{N} \in\mathbb R$ and a grid spacing constant~$d>0$ such that~$\hat s_{i+1}=\hat s_{i}+d$ for all $i = 1, \ldots, N-1$ and~$\mc N = \{\hat s_{1}, \ldots, \hat s_{N}\}$. We say that a set~$\mc M\subseteq \mathbb R$ spans the one-dimensional regular grid~$\mc N$ if~$\mc M\subseteq \mc N$, $\min\mc M=\min\mc N$ and~$\max \mc M=\max \mc N$.

\begin{theorem}[Dynamic Programming-Type Algorithm for Optimal Transport Problems with Independent Marginals]
\label{theorem:wass_app_complexity}
Suppose that $\mu = \otimes_{k \in \mathcal K} \mu_k$ is a product of $K$~independent univariate distributions of the form $\mu_k = \sum_{l \in \mc L} \mu_k^l \delta_{x_k^l}$ and that $\nu_t = t \delta_{\bm y_1} + (1 - t) \delta_{\bm y_2}$ is a two-point distribution. If $c(\bm x, \bm y) = \| \bm x - \bm y \|^2$~and if~$\mc M=\{ x_{k}^l ( y_{1, k} -  y_{2, k}):k \in \mc K,\,l \in \mc L\}$ {\color{black}spans} a regular one-dimensional grid~$\mc N$ with (known) cardinality~$N$, then the optimal transport distance between $\mu$ and $\nu_t$ can be computed exactly by a dynamic programming-type algorithm using~$\mc O(KL\log_2(KL)+ KLN + K^2 N^2)$ arithmetic operations. If all problem parameters are rational and representable as ratios of two integers with absolute values at most~$U$, then the bit lengths of all numbers computed by this algorithm are polynomially bounded in~$K$, $L$, $N$ and~$\log_2(U)$.
\end{theorem}

Assuming that~$\mc M$ {\color{black}spans} some regular one-dimensional grid~$\mc N$, Theorem~\ref{theorem:wass_app_complexity} establishes an upper bound on the number of arithmetic operations needed to solve the optimal transport problem with independent marginals. We will see that the proof of Theorem~\ref{theorem:wass_app_complexity} is constructive in that it develops a concrete dynamic programming-type algorithm that attains the indicated upper bound (see Algorithm~\ref{algorithm:dynamic-programming}). However, this bound depends on the cardinality~$N$ of the grid~$\mc N$, and Theorem~\ref{theorem:wass_app_complexity} does not relate~$N$ to~$K$, $L$ or~$U$. More importantly, it provides no guidelines for constructing~$\mc N$ or even proving its existence. 

\begin{remark}[Existence of~$\mc N$]
\label{rem:N}
If all support points of~$\mu$ and~$\nu$ have rational components, then a regular one-dimensional grid~$\mc N$ satisfying the assumptions of Theorem~\ref{theorem:wass_app_complexity} is guaranteed to exist. In general, however, its cardinality scales exponentially with~$K$ and~$L$, implying that the dynamic programming-type algorithm of Theorem~\ref{theorem:wass_app_complexity} is inefficient. To see this, assume that for all~$k \in \mc K$, $l \in \mc L$ and~$ j \in \{1,2\}$ there exist integers~$a_{k,l}, c_{j,k} \in \mathbb Z $ and~$b_{k,l}, d_{j,k} \in \mathbb N$ such that~$ x_k^l =a_{k,l}/b_{k,l}$ and~$y_{j,k} = c_{j,k}/ d_{j, k}$. Thus, we have
\[
    x_{k}^l ( y_{1, k} -  y_{2, k})=\frac{ a_k^l(c_{1,k} d_{2,k} - c_{2,k} d_{1,k})}{ b_{k,l} d_{1,k} d_{2,k}}\quad k\in\mc K, ~\forall l\in\mc L,
\]
which implies that all elements of~$\mc M$ can be expressed as rational numbers with common denominator $D=\prod_{k \in \mc K, l \in \mc L} b_{k,l} d_{1,k}d_{2,k}$. Clearly, $\mc M$ therefore {\color{black}spans} a regular one-dimensional grid~$\mc N$ with grid spacing constant~$d=D^{-1}$ and cardinality~$ N = D(\max \mc M - \min\mc M) + 1$. If~$U$ denotes as usual an upper bound on the absolute values of the integers~$a_{k,l}$, $b_{k,l}$, $c_{j,k}$ and~$d_{j,k}$ for all~$k\in\mc K$, $l\in\mc L$ and~$j\in\{1,2\}$, then we have~$D\leq U^{3KL}$, and all elements of~$\mc M$ have absolute values of at most~$2U^3$. The cardinality of~$\mc N$ therefore satisfies~$N \leq 4 U^{3(KL+1)} + 1$. This reasoning suggests that, in the worst case, the dynamic programming-type algorithm of Theorem~\ref{theorem:wass_app_complexity} may require up to~$\mc O(K^2U^{3(KL+1)})$ arithmetic operations.
\end{remark}

Remark~\ref{rem:N} guarantees that a regular one-dimensional grid~$\mc N$ satisfying the assumptions of Theorem~\ref{theorem:wass_app_complexity} exists whenever the input bit length of the optimal transport problem with independent marginals is finite. However, Remark~\ref{rem:N} also reveals that the algorithm of Theorem~\ref{theorem:wass_app_complexity} may be highly inefficient in general. Remark~\ref{rem:dyn-prog-tractability} below discusses special conditions under which this algorithm is of practical interest.

\begin{remark}[Efficiency of the Dynamic Programming-Type Algorithm]
\label{rem:dyn-prog-tractability}
The algorithm of Theorem~\ref{theorem:wass_app_complexity} is efficient on problem instances that display the following properties.
\begin{itemize}
    \item[\emph{(i)}] If~$\mc M$ {\color{black}spans} a regular one-dimensional grid whose cardinality~$N$ grows only polynomially with~$K$ and~$L$ but is independent of~$U$, then the number of arithmetic operations required by the algorithm of Theorem~\ref{theorem:wass_app_complexity} grows polynomially with~$K$ and~$L$ but is independent of~$U$, and the bit lengths of all numbers computed by this algorithm are polynomially bounded in~$K$, $L$ and~$\log_2(U)$. Hence, the algorithm runs in {\em strongly polynomial time} on a Turing machine.
    
    \item[\emph{(ii)}] If~$\mc M$ {\color{black}spans} a regular one-dimensional grid whose cardinality~$N$ grows polynomially with~$K$, $L$ and~$\log_2(U)$, then the number of arithmetic operations required by the algorithm of Theorem~\ref{theorem:wass_app_complexity} as well as the bit lengths of all numbers computed by this algorithm are polynomially bounded in~$K$, $L$ and~$\log_2(U)$. Hence, the algorithm runs in {\em weakly polynomial time} on a Turing machine.
    
    \item[\emph{(iii)}] If~$\mc M$ {\color{black}spans} a regular one-dimensional grid whose cardinality grows polynomially with~$K$, $L$ and~$U$ (but exponentially with~$\log_2(U)$), then the number of arithmetic operations required by the algorithm of Theorem~\ref{theorem:wass_app_complexity} grows polyonomially with~$K$, $L$ and~$U$, and the bit lengths of all numbers computed by this algorithm are polynomially bounded in~$K$, $L$ and~$\log_2(U)$. Hence, the algorithm runs in {\em pseudo-polynomial time} on a Turing machine.
\end{itemize}
\end{remark}
Before proving Theorem~\ref{theorem:wass_app_complexity}, we recall the definition of the Conditional Value-at-Risk (CVaR) by \citet{rockafellar2002conditional}. Specifically, if the random vector~$\bm x$ is governed by the probability distribution~$\mu$, then the CVaR at level~$t\in(0,1)$ of any Borel measurable loss function~$\ell(\bm x)$ is defined as
\[
    \text{CVaR}_t [\ell(\bm x)] = \inf_{\beta\in\mathbb R}~ \beta +  \frac{1}{t}\, \mathbb E_{\bm x\sim\mu}\left[\max\{\ell(\bm x)-\beta,0\}\right].
\]
Here, the minimization problem over~$\beta$ is solved by the Value-at-Risk (VaR) at level~$t$ \citep[Theorem~10]{rockafellar2002conditional}, which is defined as the left $(1-t)$-quantile of the loss distribution, that is,
$$\text{VaR}_t [\ell(x)] = \inf \left\{ \tau \in \mathbb R: \mu[\ell(x) \leq \tau] \geq 1-t \right\}.$$ 
The proof of~Theorem~\ref{theorem:wass_app_complexity} also relies on the following lemma.

\begin{lemma}[Minkowski sums of regular one-dimensional grids]
	\label{lemma:cardinality}
	If~$\mc N$ is a one-dimensional regular grid with cardinality~$N$ and grid spacing constant~$d>0$, then the $k$-fold Minkowski sum~$\sum_{i=1}^k\mc N$ of~$\mc N$ is another one-dimensional regular grid with cardinality~$k(N-1)+1$ and the same grid spacing constant~$d$.
\end{lemma}
\begin{proof}
  Any regular one-dimensional grid with cardinality~$N$ and grid spacing constant~$d>0$ is representable as the image of~$\{1,\ldots,N\}$ under the affine transformation~$f(s)=\hat s_1-d+ds$, where $\hat s_{1}$ denotes the smallest element of~$\mc N$. It is immediate to see that the $k$-fold Minkowski sum of~$\mc N$ is another one-dimensional regular grid with grid spacing constant~$d$. In addition, the cardinality of this Minkowski sum satisfies
   \begin{align*}
   \left| \sum_{i=1}^k \mc N \right| = \left|\sum_{i=1}^kf(\{1, \ldots, N\})\right|= \left|f\left(\sum_{i=1}^k\{1, \ldots, N\}\right)\right|
   =|f(\{k, \ldots, kN\})|
   =|\{k, \ldots,   kN\}|= k(N-1)+1, %\leq kN,
   \end{align*}
   where the second equality holds because~$f$ is affine and because the cardinality of any set is invariant under translations. Thus, the claim follows.
    \end{proof}
    
\begin{proof}[Proof of~Theorem~\ref{theorem:wass_app_complexity}]
    Throughout this proof we exceptionally assume that each arithmetic operation can be performed in unit time irrespective of the bit lengths of the involved operands. We emphasize that everywhere else in the paper, however, time is measured in the standard Turing machine model of computation. Throughout this proof we further set~$I=L^K$ and denote as usual by~$\bm x_i$, $i\in\mc I$, the~$I$ different support points of~$\mu$. Then, the optimal transport distance between $\mu$ and $\nu_t$ can be expressed as
    \begin{align}
	W_c(\mu,  {\nu}_ t)
	&= \min\limits_{\bm \pi \in \Pi(\mu,\nu_t)} ~ \sum_{i \in \mathcal I} \sum_{j \in \mathcal J} \left( \| \bm x_i \|^2 + \| \bm y_j \|^2 - 2 \bm x_i^\top \bm y_j \right) \pi_{ij} \notag \\[0.5ex]
	&= \mathbb E_{\bm x \sim \mu} \left[ \| \bm x \|^2 \right] + \mathbb E_{\bm y\sim \nu_t} \left[ \| \bm y \|^2 \right] - 2 \max\limits_{\bm \pi \in \Pi(\mu,\nu_t)} \sum_{i \in \mathcal I} \sum_{j \in \mathcal J} \bm x_i^\top \bm y_j \pi_{ij}. \label{eq:three:terms}
	\end{align}
	The two expectations in~\eqref{eq:three:terms} can be evaluated in $\mathcal O(KL)$ arithmetic operations because 
	\begin{align*}
	   % \label{eq:first:two:terms}
	    \mathbb E_{\bm x \sim \mu} \left[ \| \bm x \|^2 \right] =\sum_{k \in \mc K} \mathbb E_{ x_k \sim \mu_k} \left[ (x_k)^2 \right] = \sum_{k \in \mathcal K} \sum_{l \in \mc L} \mu_{k}^l (x_k^l)^2~ \text{and} ~~ \mathbb E_{\bm y \sim \nu_t} \left[ \| \bm y \|^2 \right] = t \| \bm y_1 \|^2 + (1 - t) \| \bm y_2 \|^2,
	\end{align*}
	and it is easy to verify that their bit lengths are polynomially bounded in~$K$, $L$ and~$\log_2(U)$.
	Moreover, as in the proof of Lemma~\ref{lemma:knapsack}, the maximization problem in~\eqref{eq:three:terms} simplifies to
	\begin{align}
	\max_{\bm \pi \in \Pi(\mu,\nu_t)} \sum_{i \in \mathcal I} \sum_{j \in \mathcal J} \bm x_i^\top \bm y_j \pi_{ij} 
	&= \left\{
	\begin{array}{cl} \max\limits_{\bm q_{1}, \bm q_{2} \in \R_+^I} & \DS t \sum_{i \in \mathcal I} \bm x_i^\top \bm y_1 q_{1,i} + (1-t) \sum_{i \in \mathcal I} \bm x_i^\top \bm y_2 q_{2,i} \\[3ex]
	\textrm{s.t.}& \bm 1^\top \bm q_1 = 1, ~ \bm 1^\top \bm q_2 = 1 \\[1ex]
	& t q_{1,i} + (1-t) q_{2,i} = \mu [\bm x = \bm x_i] \quad \forall i \in \mathcal I.
	\end{array}\right. \notag\\[2ex]
	&= \sum_{i \in \mathcal I} \bm x_i^\top \bm y_2 \, \mu [\bm x \!=\! \bm x_i] + \left\{
	\begin{array}{cl} \max\limits_{\bm q \in \R_+^I} & \DS \sum_{i \in \mathcal I} \bm x_i^\top (\bm y_1 - \bm y_2) q_{i} \\[3ex]
	\textrm{s.t.}& \bm 1^\top \bm q = t \\[1ex]
	& q_{i} \leq \mu [\bm x \! =\! \bm x_i] \quad \forall i \in \mathcal I, \!
	\end{array}\right. \label{eq:third:term}
	\end{align}
	where the second equality follows from the variable substitution~$\bm q \gets t \bm q_1$ and the subsequent elimination of $\bm q_2$ by using the equations $ (1-t) q_{2,i} = \mu [ \bm x = \bm x_i] - q_{i}$ for all~$i\in\mathcal I$. Observe next that the first sum in~\eqref{eq:third:term} can again be evaluated using $\mathcal O(KL)$ arithmetic operations because
	\begin{align*}
	    %\label{eq:third:term:sum}
	    \sum_{i \in \mathcal I} \bm x_i^\top \bm y_2 \, \mu [ \bm x = \bm x_i]
	    = \mathbb E_{\bm x \sim \mu} \left[ \bm x^\top \bm y_2 \right]
	    = \sum_{k \in \mc K} \mathbb E_{ x_k \sim \mu_k} \left[x_k  y_{2, k} \right]
	    = \sum_{k \in \mathcal K} \sum_{l \in \mc L}  x_{k}^l \mu_k^l y_{2,k},
	\end{align*}
	and the bit length of this sum is polynomially bounded in~$K$, $L$ and~$\log_2(U)$. For~$t=0$, the optimal value of the maximization problem in~\eqref{eq:third:term} vanishes. For~$t=1$, on the other hand, the problem's optimal solution satisfies~$q_i = \mu[\bm x = \bm x_i]$ for all~$i \in \mc I$. By using now standard arguments, one readily verifies that the corresponding optimal value can once again be computed in~$\mc O(KL)$ arithmetic operations and has polynomially bounded bit length in~$K$, $L$ and~$\log_2(U)$. In the remainder of the proof we may thus assume that~$t \in (0,1)$. To solve the maximization problem in~\eqref{eq:third:term} in this generic case, we first reformulate it as
	{\color{black}
	\begin{align}
	    \label{eq:fourth:term}
	    t\cdot \max \left\{ \sum_{i \in \mathcal I} \ell(\bm x_i) \,\mu[\bm x = \bm x_i]\, q_{i} \;:\; \DS \bm 0 \leq \bm q \leq t \cdot \bm 1, ~ \sum_{i \in \mathcal I} \mu [\bm x = \bm x_i]\, q_i = 1 \right\} %= t \cdot \text{CVaR}(\ell(\bm x))
	\end{align}
	by applying the variable substitution $q_i \gets q_i/ (t\cdot\mu[\bm x = \bm x_i])$ and defining $\ell(\bm x) = \bm x^\top (\bm y_1 - \bm y_2)$.
% 	, we may reformulate the maximization problem in~\eqref{eq:fourth:term} as 
% 	\[
% 	    t\cdot \max \left\{ \sum_{i \in \mathcal I} \ell(\bm x_i) \,\mu[\bm x = \bm x_i]\, q_{i} \;:\; \DS \bm 0 \leq \bm q \leq t \cdot \bm 1, ~ \sum_{i \in \mathcal I} \mu [\bm x = \bm x_i]\, q_i = 1 \right\} = t \cdot \text{CVaR}(\ell(\bm x)),
% 	\]
	The maximization problem in~\eqref{eq:fourth:term} is then readily recognized as the dual representation of the CVaR of~$\ell(\bm x)$ at level~$t$; see, {\em e.g.}, \cite[Example~6.16]{ref:shapiro2021lectures}. The expression~\eqref{eq:fourth:term} thus equals~$t \cdot \mathop{\text{CVaR}}_t(\ell(\bm x))$.}
	
	By assumption, there exists a one-dimensional regular grid~$\mc N$ with cardinality~$N$ such that $ x_{k}^l ( y_{1, k} -  y_{2, k}) \in{\mc N}$ for every $k \in \mc K$ and $l \in \mc L$. This readily implies that~${\color{black}\ell(\bm x_i)=}~\bm x_i^\top (\bm y_1 - \bm y_2) \in \mc N_K= \sum_{k=1}^K \mc N$.
	Assume from now on without loss of generality that~$\mc N_K=\{\hat s_{K,1},\ldots,\hat s_{K,|\mc N_K|}\}$ and that the elements of~$\mc N_K$ are sorted in ascending order, that is, $\hat s_{K,1} < \cdots < \hat s_{K,|\mc N_K|}$. Also, denote by~$n_t$ the unique index satisfying
	\begin{equation}
	    \label{eq:critical-index}
	    \sum\limits_{n=1}^{n_t} \mu[ \ell(\bm x) = \hat{s}_{K,n}] \geq 1- t> \sum\limits_{n=1}^{n_t - 1}  \mu[ \ell(\bm x) = \hat{s}_{K,n}].
	\end{equation}
	By \cite[Proposition~8]{rockafellar2002conditional}, {\color{black} the expression~\eqref{eq:fourth:term} can therefore be reformulated as}
	\begin{equation}
	    \label{eq:cvar}
	    t \cdot \text{CVaR}_t [\ell(\bm x) ] =  \left(\sum\limits_{n=1}^{n_t} \mu[\ell(\bm x) = \hat{s}_{K,n}] -(1- t) \right) \hat{s}_{K,n_{t}} + \sum\limits_{n = n_t + 1}^{|\mc N_K|} \mu[\ell(\bm x) = \hat{s}_{K,n}] \hat{s}_{K,n}.
	\end{equation}
	{\color{black} Computing~\eqref{eq:cvar} thus amounts to evaluating a sum of~$\mathcal O(|\mc N_K|)$ terms.} We will now prove that evaluating this sum requires~{\color{black}$\mathcal O(KL \log_2(KL) + K L N + K^2 N^2)$} arithmetic operations. 
	To this end, we first show that the grid points~$\hat s_{K,n}$, $n=1,\ldots,|\mc N_K|$, can be computed in time~{\color{black}$\mc O(KL \log_2(KL)+KN)$} (Step~1), then we show that the probabilities~$\mu[\ell(\bm x) = \hat s_{K,n}]$, $n=1,\ldots, |\mc N_K|$, can be computed recursively in time~{\color{black}$\mc O(KLN+K^2N^2)$} (Step~2), and finally we use these ingredients to compute the right hand side of~\eqref{eq:cvar} in time~$\mc O(KN)$ (Step~3). 

	{\em Step~1.} By assumption, the regular grid~$\mc N$ has known cardinality~$N$ {\color{black} and is spanned by}~$\mc M=\{ x_{k}^l ( y_{1, k} -  y_{2, k}):k \in \mc K,\,l \in \mc L\}$. To compute all elements of~$\mc N$, we first compute all elements of~$\mc M$ in time~$\mc O(KL)$ and sort them in non-decreasing order in time~$\mc O(KL\log_2(KL))$ using merge sort, for example. As~$\mc M$ spans~$\mc N$, the minimum and the maximum of~$\mc M$ coincide with the minimum~$\hat s_1$ and the maximum~$\hat s_N$ of~$\mc N$, respectively. Given~$\hat s_1$ and~$\hat s_N$, we can then compute the grid spacing constant~$d=(\hat s_N-\hat s_1)/(N-1)$ as well as the elements~$\hat s_n=\hat s_1+d(n-1)$, $n=1,\ldots,N$, of~$\mc N$, which requires~$\mc O(N)$ arithmetic operations. The bit lengths of all numbers computed so far are bounded by a polynomial in~$\log_2(U)$ and~$\log_2(N)$.
	
	It is easy to see that $\mc N_K=\sum_{k=1}^K\mc N$ is also a one-dimensional regular grid that has the same grid spacing constant as~$\mc N$ and whose minimum~$\hat s_{K,1}=K\hat s_1$ can be computed in constant time. The elements of~$\mc N_K$ are then obtained by computing~$\hat s_{K,n}=\hat s_{K,1}+d(n-1)$ for all~$n=1,\ldots,|\mc N_K|$, where~$|\mc N_K|=K(N-1)+1$ thanks to Lemma~\ref{lemma:cardinality}. This computation requires~$\mc O(KN)$ arithmetic operations, and the bit lengths of all involved numbers are still bounded by a polynomial in~$\log_2(U)$ and~$\log_2(N)$. This completes Step~1. 
	
    %Therefore, if we had access to the probabilities $\mu [ \ell(\bm x) = \hat{s}_{Kn} ]$ for all $n =1, \ldots, |\mc N_K|$, we could compute $\text{CVaR}_t [\ell(\bm x) ]$ in time ${\mathcal O}(\mc S(|\mc N_K|)) = \mc O((\mc S(NK)))$ thanks to Lemma~\ref{lemma:cardinality}.

    {\em Step~2.} We now show that the probabilities~$\mu[\ell(\bm x) = \hat{s}_{K,n}]$ for~$n=1,\ldots,|\mc N_K|$ can be calculated recursively in time $\mathcal O(K^2 N^2)$.
	To this end, we introduce the partial sums~$\ell_k(\bm x) = \sum_{m=1}^k x_m (y_{1,m} - y_{2,m})$ for every~$k \in \mathcal K$ and note that $\ell_K(\bm x)=\ell(\bm x)$. For every~$k \in \mc K$, the range of the function~$\ell_k(\bm x)$ is a subset of the one-dimensional regular grid~$\mc N_k=\sum_{k'=1}^k\mc N$.
	The law of total probability then implies that
	\begin{align*}
	    \mu [\ell_k(\bm x) = \hat{s}]
	    &= \sum\limits_{\hat s' \in \mc N}\mu \left[ \ell_{k-1}(\bm x) = \hat{s} -\hat{s}',~ x_k (y_{1,k} - y_{2,k}) = \hat{s}' \right] \quad \forall k \in \mathcal K\backslash \{1\},~\forall \hat s\in\mc N_k,
	\end{align*}
	where $\hat s_{1}, \ldots, \hat s_{N}$ denote as usual the elements of~$\mc N$, 
% 	and where $\mu[\ell_1(\bm x) = \hat s] = \sum_{i=1}^N \mu_k[x_1(y_{1,k} - y_{2,k}) = \hat s_i]$ for all~$\hat s \in \mc N_1$.
	{\color{black} 
	and where $\mu[\ell_1(\bm x) = \hat s] = \mu_1[x_1(y_{1, 1} - y_{2, 1}) = \hat s]$ for all $\hat s \in \mc N_1$.}
    As $\ell_k(\bm x)=\ell_{k-1}(\bm x)+ x_k (y_{1,k} - y_{2,k})$, 
    $\ell_{k-1}(\bm x)$ is constant in $x_k,\ldots,x_K$ and the components of~$\bm x$ are mutually independent under the product distribution~$\mu=\otimes_{k\in\mc K}\mu_k$, we thus have
	\begin{align}
		\label{eq:recursion}
	    \mu [\ell_k(\bm x) = \hat{s}] = \sum\limits_{\hat s' \in \mc N} \mu \left[ \ell_{k-1}(\bm x) = \hat{s} -\hat{s}'\right]\times \mu_k\left[ x_k (y_{1,k} - y_{2,k}) = \hat{s}'\right]\quad \forall k \in \mc K \backslash \{1\},~\forall \hat s \in \mc N_k.
	\end{align}
    The marginal probabilities~$\mu_k [x_k (y_{1,k} - y_{2,k}) = \hat{s}']$ for all $k \in \mc K$ and~$\hat s' \in \mc N$ can be pre-computed in time~$\mc O(KLN)$. Given~$\mu[\ell_{k-1}(\bm x) = \hat s]$, $\hat s \in \mc N_{k-1}$, each probability~$\mu[\ell_k(\bm x) = \hat{s}]$, $\hat s\in\mc N_k$, can then be computed in time~$\mc O(N)$ by using~\eqref{eq:recursion}. As $|\mc N_k|=\mc O(kN)$ for every~$k\in\mc K$ thanks to Lemma~\ref{lemma:cardinality}, each iteration~$k\in\mc K$ of the  the dynamic programming-type recursion~\eqref{eq:recursion} requires at most~$\mc O(KN^2)$ arithmetic operations. Finally, as there are~$\mc O(K)$ iterations in total, the sought probabilities~$\mu[\ell_{K}(\bm x) = \hat s]$, $\hat s\in\mc N_K$, can be computed in time~$\mc O(K^2N^2)$. An elementary calculation further shows that the bit lengths of these probabilities are bounded by a polynomial in~$K$, $N$ and~$\log_2(U)$. This completes Step~2.
    
    {\em Step~3.} As all terms appearing in the sum on the right hand side of~\eqref{eq:cvar} have been pre-computed in Steps~1 and~2, the sum itself can now be evaluated in time~$\mc O(KN)$ thanks to Lemma~\ref{lemma:cardinality}. Note that the critical index~$n_t$ defined in~\eqref{eq:critical-index} can also be computed in time~$\mc O(KN)$. The bit lengths of all numbers involved in these calculations are bounded by a polynomial in~$K$, $N$ and~$\log_2(U)$. This completes Step~3.
    
    In summary, the time required for evaluating the CVaR in~\eqref{eq:cvar} totals~$\mc O(KL\log_2(KL)+KLN+K^2N^2)$, which matches the overall time required for all calculations described in Steps~1, 2 and~3. This computation time dominates the time~$\mc O(KL)$ spent on all preprocessing steps, and thus the claim follows.
\end{proof}

The dynamic programming-type procedure developed in the proof of Theorem~\ref{theorem:approx-hard} is summarized in Algorithm~\ref{algorithm:dynamic-programming}. This procedure outputs the optimal transport distance between~$\mu$ and~$\nu_t$ (denoted by~$W_c$). In addition, Algorithm~\ref{algorithm:dynamic-programming} can be used for constructing the optimal transportation plan from~$\mu$ to~$\nu_t$.

\begin{table}[H]
\centering
\begin{minipage}{0.8\textwidth}
\vspace{-1em}
\begin{algorithm}[H]
\caption{Optimal Transport with Independent Marginals \label{algorithm:dynamic-programming}}
	\begin{algorithmic}[1]
		\Require $\{\mu_k^l\}_{k \in \mc K, l \in \mc L}$, ~$\{x_k^l\}_{k\in \mc K, l\in \mc L}$, ~$\bm y_1 , \bm y_2 \in \R^K$,~ $t$, ~$N$
		\State Initialize~$\hat s_1 = \min\limits_{k \in \mc K, l \in \mc L} x_k^l (y_{1,k} - y_{2,k})$ and $\hat s_N = \max\limits_{k \in \mc K, l \in \mc K} x_k^l (y_{1,k} - y_{2,k})$
		\State Set~$d = (\hat s_N - \hat s_1) / (N-1)$ and~$\hat s_n = \hat s_1 + d(n-1)~\forall n=1, \ldots, N$
        \State Compute $\mu_k[x_k(y_{1,k} - y_{2,k}) = \hat s_n]$~$\forall k \in \mc K$ and $n \in \mc N$
        % \State Set $\mu[\ell_1(\bm x) = \hat s_{1, n}] =\sum\limits_{\hat s' \in \mc N} \mu_1[x_1(y_{1,1} - y_{2,1}) = \hat s ']~\forall n =1, \ldots, N $
        \State {\color{black}Set~$\mu[\ell_1 (\bm x) = \hat s_{n}] = \mu_1[x_1(y_{1,1} - y_{2,1}) = \hat s_n]~\forall n=1,\ldots, N$ }
		\For{$k = 2,\ldots, K$}
		\For{$n= 1,\ldots, k (N-1) + 1$}
		\State $\hat s_{k, n} = k\hat s_1+ d(n-1)$
		\State $\mu[\ell_k(\bm x)  = \hat s_{k, n}] = \sum\limits_{\hat s' \in \mc N} \mu[\ell_{k-1}(\bm x) = \hat s_{k, n} - \hat s']\times \mu_k[  x_k(y_{1,k} -y_{2,k}) = \hat s'] $
		\EndFor
		\EndFor
		\State Find the index~$n_t \in \{1, \ldots, K(N-1) + 1\}$ satisfying~\eqref{eq:critical-index}
		\State Set~$${\textrm{CVaR}} =\frac{1}{t}\left[\left(\sum_{n=1}^{n_t} \mu[\ell_K(\bm x)\! =\! \hat s_{K, n}] \!-\! 1\! +\! t\right) \hat s_{K, n_t}\! -\!2 \sum_{n= n_t + 1}^{K(N-1) + 1}\! \mu[\ell_K(\bm x) \!=\! \hat s_{K, n}] \hat s_{K, n}\right] $$
		\State Set $$W_c \!=	\sum\limits_{k\in\mc K}\sum\limits_{l\in\mc L} \mu_k^l (x_k^l)^2\! +\!t \sum\limits_{k\in \mc K} y_{1,k}^2 \!+\! (1\!-\!t) \sum\limits_{k\in\mc K} y_{2,k}^2 -2 \sum\limits_{k\in \mc K} \sum\limits_{l \in \mc L} x_k^l \mu_k^l y_{2,k}-2 t \cdot{\textrm{CVaR}} $$ 
		\Ensure $W_c$
	\end{algorithmic}
\end{algorithm}
\vspace{-1em}
\end{minipage}
\end{table}

\begin{remark}[Optimal Transportation Plan]
The critical index~$n_t$ computed by Algorithm~\ref{algorithm:dynamic-programming} allows us to construct an optimal transportation plan~$\bm\pi^\star\in\mathbb R^{I\times J}_+$ that solves the linear program~\eqref{eq:primal}, where~$\pi^\star_{i,j}$ denotes the probability mass moved from~$\bm x_i$ to~$\bm y_j$ for every~$i\in\mathcal I$ and~$j \in \mathcal J$. To see this, note that the defining properties of~$n_t$ in~\eqref{eq:critical-index} imply that~$\mathrm{VaR}_t[\ell(\bm x)] = \hat s_{K, n_t}$ and~$\mu[\ell(\bm x) = \hat s_{K, n_t}]>0$. We may thus define~$\bm\pi^\star$ via
\begin{align*}
    \pi^\star_{i,1} = 
    \begin{cases}
    \mu[\bm x = \bm x_i] & \text{if } \ell(\bm x_i) > \hat s_{K, n_t} \\[3ex]
    \displaystyle \frac{t - 1 + \sum_{n=1}^{n_t} \mu[\ell(\bm x) = \hat s_{K, n}]}{\mu[\ell(\bm x) = \hat s_{K, n_t}]} \times \mu[\bm x = \bm x_i] \qquad & \text{if } \ell(\bm x_i) = \hat s_{K, n_t}\\[3ex]
    0 & \text{if } \ell(\bm x_i) < \hat s_{K, n_t} 
    \end{cases}
\end{align*}
and $\pi^\star_{i,2} = \mu[\bm x = \bm x_i] - \pi^\star_{i,1}$ for all~$i \in \mathcal I$. By the first inequality in~\eqref{eq:critical-index}, we have~$\bm \pi^\star\geq \bm 0$. In addition, we trivially find~$\pi^\star_{i,1}+\pi^\star_{i,2}=\mu[\bm x = \bm x_i]$ for all~$i\in\mathcal I$, and we have
\begin{align*}
      \sum_{i\in\mathcal I}\pi^\star_{i,1}&=\sum_{\substack{i\in\mathcal I:\\ \ell(\bm x_i)>\hat s_{K,n_t}}}\mu[\bm x = \bm x_i] + \sum_{\substack{i\in\mathcal I:\\ \ell(\bm x_i)=\hat s_{K,n_t}}}\frac{t - 1 + \sum_{n=1}^{n_t} \mu[\ell(\bm x) = \hat s_{K, n}]}{\mu[\ell(\bm x) = \hat s_{K, n_t}]} \times \mu[\bm x = \bm x_i] \\
      &=\sum_{n=n_t+1}^{|\mathcal N_K|} \mu[\ell(\bm x) = \hat s_{K, n}] + t - 1 + \sum_{n=1}^{n_t} \mu[\ell(\bm x) = \hat s_{K, n}]= t = 1-\sum_{i\in\mathcal I}\pi^\star_{i,2}.
\end{align*}
In summary, this shows that~$\bm \pi^\star$ is feasible in the optimal transport problem~\eqref{eq:primal}. Finally, we have
\begin{align*}
      \sum_{i \in \mathcal I} \sum_{j \in \mathcal J} \pi^\star_{ij} \| \bm x_i - \bm y_j \|^2
      &= \mathbb E_{\bm x \sim \mu} \left[ \| \bm x \|^2 \right] + \mathbb E_{\bm y\sim \nu_t} \left[ \| \bm y \|^2 \right] - 2 \sum_{i \in \mathcal I} \sum_{j \in \mathcal J} \bm x_i^\top \bm y_j \pi_{ij}^\star \\
      %&= \mathbb E_{\bm x \sim \mu} \left[ \| \bm x \|^2 \right] + \mathbb E_{\bm y\sim \nu_t} \left[ \| \bm y \|^2 \right] - 2 \; \mathbb E_{\bm x \sim \mu} \left[ \bm x^\top \bm y_2 \right] - 2 \sum_{i \in \mathcal I} \sum_{j \in \mathcal J} \bm x_i^\top (\bm y_1 - \bm y_2) \pi_{ij}^\star \\
      &= \mathbb E_{\bm x \sim \mu} \left[ \| \bm x \|^2 \right] + \mathbb E_{\bm y\sim \nu_t} \left[ \| \bm y \|^2 \right] - 2 \; \mathbb E_{\bm x \sim \mu} \left[ \bm x^\top \bm y_2 \right] - 2 \sum_{i \in \mathcal I} \ell(\bm x_i)\, \pi_{i,1}^\star \\
      &= \mathbb E_{\bm x \sim \mu} \left[ \| \bm x \|^2 \right] + \mathbb E_{\bm y\sim \nu_t} \left[ \| \bm y \|^2 \right] - 2 \; \mathbb E_{\bm x \sim \mu} \left[ \bm x^\top \bm y_2 \right] - 2 t \cdot \text{CVaR}_t [\ell(\bm x)],
\end{align*}
where the first two equalities follow from~\eqref{eq:three:terms} and~\eqref{eq:third:term}, respectively, while the third equality exploits the definitions of~$\bm \pi^\star$ and the CVaR. The last expression manifestly matches the output~$W_c$ of Algorithm~\ref{algorithm:dynamic-programming}. Hence, we may conclude that $\bm \pi^\star$ is indeed optimal in~\eqref{eq:primal}. Note that evaluating~$\pi^\star_{ij}$ for a fixed~$i\in\mathcal I$ and~$j\in\mathcal J$ requires at most~$\mathcal O(N K + KL)$ arithmetic operations provided that the critical index $n_t$ and the probabilities $\mu[\ell(\bm x) = \hat s_{K, n}]$, $n \in \mathcal N_K$, are given. These quantities are indeed computed by Algorithm~\ref{algorithm:dynamic-programming}.  
%in~$\mathcal O(K^2 L^2 + K^2 N^2)$ arithmetic operations.
\end{remark}

In the following we will identify special instances of the optimal transport problem with independent marginals that can be solved efficiently. Assume first that both~$\mu$ and~$\nu$ are supported on~$\{0,1\}^K$. This implies that all marginals of~$\mu$ represent independent Bernoulli distributions. Unlike in Theorem~\ref{theorem:approx-hard}, however, these Bernoulli distributions may be non-uniform. The following corollary shows that, in this case, the optimal transport problem with independent marginals can be solved in strongly polynomial time.

\begin{corollary}[Binary Support]
\label{corollary:tractable_example0} Suppose that all assumptions of Theorem~\ref{theorem:wass_app_complexity} hold. If in addition~$L=2$, $x^1_k=0$ and~$x^2_k=1$ for all~$k\in\mc K$, and~$\bm y_1, \bm y_2 \in \{0, 1\}^K$, then the optimal transport distance between~$\mu$ and~$\nu_t$ can be computed in strongly polynomial time.
\end{corollary}

\begin{proof}
Under the given assumptions, we have~$\mc M=\{x_k^l (y_{1,k} - y_{2,k}):k\in \mc K,\, l\in \mc L\}\subseteq \{-1,0,1\}$. Hence, Theorem~\ref{theorem:wass_app_complexity} applies with~${\mc N}\subseteq \{ -1, 0, 1 \}$ and~$N\leq 3$, and therefore Algorithm~\ref{algorithm:dynamic-programming} computes~$W_c(\mu,\nu_t)$ using~$\mc O(K^2)$ arithmetic operations. As~$N$ is constant in~$K$, $L$ and~$\log_2(U)$, Remark~\ref{rem:dyn-prog-tractability}\,\emph{(i)} implies that~$W_c(\mu,\nu_t)$ can be computed in strongly polynomial time in the Turing machine model.
\end{proof}

By generalizing the proof of Corollary~\ref{corollary:tractable_example0} in the obvious way, one can show that the optimal transport problem with independent marginals remains strongly polynomial-time solvable whenever~$\mu$ and~$\nu_t$ are supported on a (fixed) bounded subset of the scaled integer lattice~$\mathbb Z^K/M$ for some (fixed) scaling factor~$M\in\mathbb N$. If~$\mu$ and~$\nu_t$ are supported on a subset of~$\mathbb Z^K/M$ that may grow with the problem's input size or if the scaling factor~$M$ may grow with the input size, then Algorithm~\ref{algorithm:dynamic-programming} ceases to run in polynomial time. We now show, however, that Algorithm~\ref{algorithm:dynamic-programming} stills run in pseudo-polynomial time in these cases.

\begin{corollary}[Lattice Support]
\label{corollary:tractable_example1} Suppose that all assumptions of Theorem~\ref{theorem:wass_app_complexity} hold. If there exists a positive integer~$M\leq U$, such that~$x^l_k\in\mathbb Z/M$ for all~$k\in\mc K$ and~$l\in\mc L$, while~$\bm y_1,\bm y_2\in\mathbb Z^K/M$, then the optimal transport distance between~$\mu$ and~$\nu_t$ can be computed in pseudo-polynomial time.
\end{corollary}

\begin{proof}
    Under the given assumptions, we have~$\mc M=\{x_k^l (y_{1,k} - y_{2,k}):k\in \mc K,\, l\in \mc L\}\subseteq \mathbb Z/M^2$. Therefore, $\mc M$ spans a one-dimensional regular grid~$\mc N\subseteq \mathbb Z/M^2$ with grid spacing constant~$d=1/M^2$ and cardinality
    \begin{equation}
    \label{eq:N-formula}
    \begin{aligned}
        N&=\left(\max\mc M -\min\mc M\right)/d \\
        & = \max_{k\in\mc K,\,l\in\mc L}\left\{ Mx_k^l (My_{1,k} - My_{2,k})\right\} -\min_{k\in\mc K,\,l\in\mc L} \left\{Mx_k^l (My_{1,k} - My_{2,k}) \right\}.
    \end{aligned}
    \end{equation}
    Recall that~$x_k^l=a_k^l/b_k^l$ for some~$a_k^l\in\mathbb Z$ and~$b_k^l\in\mathbb N$ with~$|a_k^l|,|b_k^l|\leq U$ and that~$M\leq U$. We may thus conclude that~$|Mx^l_k|\leq U^2$ for all~$k\in\mc K$ and~$l\in\mc L$. Similarly, one can show that~$|My_{1,k}|\leq U^2$ and~$|My_{2,k}|\leq U^2$ for all~$k\in\mc K$. By~\eqref{eq:N-formula}, we thus have~$N\leq 4U^2$, which implies via Theorem~\ref{theorem:wass_app_complexity} that Algorithm~\ref{algorithm:dynamic-programming} computes~$W_c(\mu,\nu_t)$ using~$\mc O(KL\log_2(KL)+K^2U^4)$ arithmetic operations. We emphasize that the number of arithmetic operations thus grows polynomially with~$K$, $L$ and~$U$ but exponentially with~$\log_2(U)$. By Remark~\ref{rem:dyn-prog-tractability}\,\emph{(iii)}, $W_c(\mu,\nu_t)$ can therefore be computed in pseudo-polynomial time.
\end{proof}

So far we have discussed methods to solve the optimal transport problem with independent marginals {\em exactly}. In the remainder of this section we will show that {\em approximate} solutions can always be computed in pseudo-polynomial time. The following lemma provides a key ingredient for this argument.

\begin{lemma}[Approximating Optimal Transport Distances]
\label{lem:approximate-OT}
Consider four discrete probability distributions~$\mu= \sum_{i \in \mathcal I} \mu_i \delta_{\bm x_i}$, $\tilde\mu= \sum_{i \in \mathcal I} \mu_i \delta_{\tilde{\bm x}_i}$, $\nu= \sum_{j \in \mathcal J} \nu_j \delta_{\bm y_j}$ and~$\tilde\nu= \sum_{j \in \mathcal J} \nu_j \delta_{\tilde{\bm y}_j}$ supported on a hypercube~$[-U,U]^K$ for some~$U>0$. If $c(\bm x, \bm y) = \| \bm x - \bm y \|^2$ and there exists~$\varepsilon\ge 0$ such that $\|\tilde{\bm x}_i-\bm x_i\|_\infty\leq \varepsilon$ for all~$i\in\mc I$ and~$\|\tilde{\bm y}_j-\bm y_j\|_\infty\leq \varepsilon$ for all~$j\in\mc J$, then we have
\begin{align}
\label{eq:diff_wass_ineq}
    |W_c(\mu, \nu) - W_c(\tilde \mu, \tilde \nu)|  \leq 8KU\varepsilon.
\end{align}
\end{lemma}

We emphasize that Lemma~\ref{lem:approximate-OT} holds for arbitrary discrete distributions~$\mu$, $\tilde\mu$, $\nu$ and~$\tilde\nu$ provided that~$\tilde\mu$ and~$\tilde \nu$ are obtained by perturbing only the support points of~$\mu$ and~$\nu$, respectively, but not the corresponding probabilities. In particular, the lemma holds even if~$\mu$ and~$\tilde \mu$ fail to represent product distributions with independent marginals, and even if~$\nu$ and~$\tilde \nu$ fail to represent two-point distributions. Note also that, by slight abuse of notation, $\mu_i$, $i\in\mc I$, represent here the probabilties of the support points of~$\mu$ and should not be confused with the univariate marginal distributions~$\mu_k$, $k\in\mc K$, in the rest of the paper.

\begin{proof}[Proof of Lemma~\ref{lem:approximate-OT}]
The elementary identity~$|a^2 - b^2 | = (a+b)|a-b|$ for any~$a, b \in \R_+$ implies that
\begin{align}
\label{eq:diff_wass_ineq1}
    |W_c(\mu, \nu) - W_c(\tilde \mu, \tilde \nu)|  &= \left( W_c( \mu,  \nu_t)^{\frac{1}{2}}  + W_c(\tilde \mu, \tilde \nu)^{\frac{1}{2}}  \right)\left| W_c(\mu, \nu)^{\frac{1}{2}} - W_c(\tilde \mu,\tilde \nu)^{\frac{1}{2}}  \right|.
\end{align}
By the definition of the optimal transport distance, the first term on the right-hand-side of~\eqref{eq:diff_wass_ineq1} satisfies 
\begin{align*}
    W_c(\mu, \nu)^{\frac{1}{2}} + W_c(\tilde \mu, \tilde \nu)^{\frac{1}{2}} &= \left(\min\limits_{\pi \in \Pi(\mu, \nu)} \sum\limits_{i \in \mc I}\sum\limits_{j \in \mc J} \| \bm x_i - \bm y_j\|^2 \pi_{ij} \right)^{\frac{1}{2}} + \left(\min\limits_{\tilde\pi \in \Pi(\tilde \mu, \tilde \nu)} \sum\limits_{i \in \mc I}\sum\limits_{j \in \mc J} \| \tilde{\bm x}_i - \tilde{\bm y}_j\|^2 \tilde{\pi}_{ij}\right)^{\frac{1}{2}} \\
    &\leq 4 \sqrt{K} U,
\end{align*}
where the inequality holds because~$\pi$ and~$\tilde\pi$ are probability distributions and because
\[
    \|\bm x_i - \bm y_j\|^2\leq K \|\bm x_i - \bm y_j\|_\infty^2\leq 4KU^2\quad\text{and}\quad \|\tilde{\bm x}_i - \tilde{\bm y}_j\|^2\leq \|\tilde{\bm x}_i - \tilde{\bm y}_j\|_\infty^2\leq4KU^2
\]
for all~$i \in \mc I$ and~$j \in \mc J$, taking into account that all support points of the four probability distributions~$\mu$, $\tilde\mu$, $\nu$ and~$\tilde\nu$  fall into the hypercube~$[-U,U]^K$. The second term on the right-hand-side of~\eqref{eq:diff_wass_ineq1} satisfies
\begin{align*}
        \left|W_c(\mu, \nu)^{\frac{1}{2}} - W_c(\tilde\mu, \tilde \nu)^{\frac{1}{2}} \right|&\leq \left|W_c(\mu, \nu)^{\frac{1}{2}} - W_c(\tilde \mu, \nu)^{\frac{1}{2}} \right|+\left|W_c(\tilde \mu, \nu)^{\frac{1}{2}} - W_c(\tilde\mu, \tilde \nu)^{\frac{1}{2}} \right| \\
        & \leq W_c(\mu, \tilde \mu)^{\frac{1}{2}} + W_c(\nu, \tilde \nu)^{\frac{1}{2}}\\
        &= \left(\min\limits_{\pi^\mu \in \Pi(\mu, \tilde \mu)} \sum\limits_{i , i'\in \mc I}\| \bm x_i - \tilde{\bm x}_{i'}\|^2 \pi^\mu_{ii'}\right)^{\frac{1}{2}}+\left(\min\limits_{\pi^\nu \in \Pi(\nu, \tilde \nu)} \sum\limits_{j , j'\in \mc J}\| \bm y_j - \tilde{\bm y}_{j'}\|^2 \pi^\nu_{jj'}\right)^{\frac{1}{2}} \\
        &\leq \left(\frac{1}{I}\sum\limits_{i \in \mc I} \| \bm x_i - \tilde{\bm x}_{i}\|^2 \right)^{\frac{1}{2}} +\left(\frac{1}{J}\sum\limits_{j \in \mc J} \| \bm y_j - \tilde{\bm y}_{j}\|^2 \right)^{\frac{1}{2}} \leq2\sqrt{K}\varepsilon,
\end{align*}
where the second inequality holds because the 2-Wasserstein distance is a metric and thus obeys the triangle inequality~\cite[\S~6]{villani}, whereas the third inequality holds because~$\pi^\mu$ and~$\pi^\nu$ with~$\pi^\mu_{ii'}=\frac{1}{I}\delta_{ii'}$ for all~$i,i'\in\mc I$ and~$\pi^\nu_{jj'}=\frac{1}{J}\delta_{jj'}$ for all~$j,j'\in\mc J$, respectively, are feasible transportation plans. Finally, the last inequality follows from our assumption that~$\|\bm x_i -\tilde{\bm x}_i\|_\infty\leq \varepsilon$ and~$\|\bm y_j -\tilde{\bm y}_j\|_\infty\leq \varepsilon$, which implies that
\[
    \| \bm x_i - \tilde{\bm x}_{i}\|^2\leq K \| \bm x_i - \tilde{\bm x}_{i}\|_\infty^2\leq K\varepsilon^2 \quad\text{and}\quad \| \bm y_j - \tilde{\bm y}_{j}\|^2\leq K \| \bm y_j - \tilde{\bm y}_{j}\|_\infty^2\leq K\varepsilon^2
\]
for all~$i\in\mc I$ and ~$j\in\mc J$. Substituting the above estimates back into~\eqref{eq:diff_wass_ineq1} finally yields~\eqref{eq:diff_wass_ineq}.
\end{proof}

We now address the approximate solution of optimal transport problems with independent marginals. 

\begin{theorem}[Approximate Solutions of the Optimal Transport Problem with Independent Marginals]
\label{theorem:approximation}
Suppose that $\mu = \otimes_{k \in \mathcal K} \mu_k$ with $\mu_k = \sum_{l \in \mc L} \mu_k^l \delta_{x_k^l}$ for every~$k\in\mc K$ and that $\nu_t = t \delta_{\bm y_1} + (1 - t) \delta_{\bm y_2}$, and let~$\varepsilon>0$ be an error tolerance. If $c(\bm x, \bm y) = \| \bm x - \bm y \|^2$ and if all probabilities and coordinates of the support points of~$\mu$ and~$\nu_t$ are representable as fractions of two integers with absolute values of at most~$U$, then the optimal transport distance between $\mu$ and $\nu_t$ can be computed to within an absolute error of at most~$\varepsilon$ by a dynamic programming-type algorithm using~$\mc O(KL\log_2(KL)+K^6U^8/\varepsilon^4)$ arithmetic operations. The bit lengths of all numbers computed by this algorithm are polynomially bounded in~$K$, $L$, $\log_2(U)$ and~$\log_2(\frac{1}{\varepsilon})$.
\end{theorem}

\begin{proof}
In order to approximate~$W_c(\mu,\nu_t)$ to within an absolute accuracy of~$\varepsilon$, we define~$M=\lceil 8KU/\varepsilon\rceil$ and map all support points of~$\mu$ and~$\nu$ to the nearest lattice points in~$\mathbb Z^K/M$ to construct perturbed probability distributions~$\tilde\mu$ and~$\tilde \nu_t$, respectively. Specifically, we construct~$\tilde x^l_k$ by rounding~$x^l_k$ to the nearest point in~$\mathbb Z/M$ for every~$k\in\mc K$ and~$l\in\mc L$. This requires~$\mc O(KL)$ arithmetic operations. We then define the perturbed marginal distributions~$\tilde\mu_k=\sum_{l\in\mc L} \mu^l_k\delta_{\tilde x^l_k}$ for all~$k\in\mc K$ and set~$\tilde\mu=\otimes_{k\in\mc K}\tilde\mu_k$. In addition, we denote by~$\tilde{\bm x}_i$, $i\in\mc I$, the~$I$ different support points of~$\tilde \mu$. Here, it is imperative to use the same orderings for the support points of~$\mu$ and~$\tilde \mu$, which implies that~$\|\bm x_i-\tilde{\bm x}_i\|_\infty\leq\frac{1}{M}\leq\frac{\varepsilon}{8KU}$ for all~$i\in\mc I$ thanks to the construction of~$\tilde \mu$. We further construct~$\tilde y_{j,k}$ by rounding~$y_{j,k}$ to the nearest points in~$\mathbb Z/M$ for every~$j\in\mc J=\{1,2\}$ and~$k\in\mc K$, and we define~$\tilde{\bm y}_j=(\tilde y_{j,1},\ldots, \tilde y_{j,K})$ for all~$j\in\mc J$. This construction requires~$\mc O(K)$ arithmetic operations and guarantees that~$\|\bm y_j-\tilde{\bm y}_j\|_\infty\leq\frac{1}{M}\leq \frac{\varepsilon}{8KU}$ for all~$j\in\mc J$. Finally, we introduce the perturbed two-point distribution~$\tilde \nu_t=t\delta_{\tilde{\bm y}_1}+(1-t)\delta_{\tilde{\bm y}_2}$. All support points of~$\mu$ and~$\nu$ have rational coordinates that are representable as fractions of two integers with absolute values at most~$U$. Therefore, $\mu$ and~$\nu$ are supported on~$[-U,U]^K$. Similarly, as~$U$ and~$M$ are integers, which implies that~$U$ is an integer multiple of~$\frac{1}{M}$, and as all support points of~$\tilde \mu$ and~$\tilde\nu$ are obtained by mapping the support points of~$\mu$ and~$\nu$ to the nearest lattice points in~$\mathbb Z^K/M$, respectively, the perturbed distributions~$\tilde \mu$ and~$\tilde\nu$ must also be supported on~$[-U,U]^K$. Lemma~\ref{lem:approximate-OT} therefore guarantees that~$|W_c(\mu, \nu_t) - W_c(\tilde \mu, \tilde \nu_t)|  \leq \varepsilon$.

In the remainder of the proof we will estimate the number of arithmetic operations needed to compute~$W_c(\tilde \mu, \tilde \nu_t)$. Note first that the coordinates of all support points of~$\tilde\mu$ and~$\tilde \nu_t$ are fractions of integers with absolute values of at most~$\tilde U=MU$. To see this, recall that~$x_k^l=a_k^l/b_k^l$ for some~$a_k^l\in\mathbb Z$ and~$b_k^l\in\mathbb N$ with~$|a_k^l|,|b_k^l|\leq U$. Using `round' to denote the rounding operator that maps any real number to its nearest integer, we can express~$\tilde x^l_k$ as~$\tilde a_k^l/\tilde b_k^l$ with~$\tilde a_k^l=\mathop{\text{round}}(Mx^l_k)\in\mathbb Z$ and~$\tilde b^l_k=M\in\mathbb N$. By construction, we have~$|\tilde a_k^l|\leq MU=\tilde U$ and~$\tilde b_k^l=M\leq\tilde U$ for all~$k\in\mc K$ and~$l\in\mc L$. Similarly, one can show that~$\tilde y_{j,k}$ is representable as a fraction of two integers with absolute values of at most~$\tilde U$ for all~$j\in\mc J$ and~$k\in\mc K$. As~$M\leq\tilde U$, $\tilde\mu$ and~$\tilde\nu$ thus satisfy all assumptions of Corollary~\ref{corollary:tractable_example1} with~$\tilde U$ instead of~$U$, respectively. From the proof of this corollary we may therefore conclude that~$W_c(\tilde \mu, \tilde \nu_t)$ can be computed in~$\mc O(KL\log_2(KL)+K^2\tilde U^4)$ arithmetic operations using Algorithm~\ref{algorithm:dynamic-programming}. As~$\tilde U=MU=\mc O(KU^2/\varepsilon)$, this establishes the claim about the number of arithmetic operations. From the definitions of~$\tilde U$ and~$M$ and from the analysis of Algorithm~\ref{algorithm:dynamic-programming} in Theorem~\ref{theorem:wass_app_complexity}, it is clear that the bit lengths of all numbers computed by the proposed procedure are indeed polynomially bounded in~$K$, $L$, $\log_2(U)$ and~$\log_2(\frac{1}{\varepsilon})$. This observation completes the proof.
\end{proof}

Theorem~\ref{theorem:approximation} shows that an $\varepsilon$-approximation of~$W_c(\mu, \nu_t)$ can be computed with a number of arithmetic operations that grows only polynomially with~$K$, $L$, $U$ and~$\frac{1}{\varepsilon}$ but exponentially with~$\log_2(U)$ and~$\log_2(\frac{1}{\varepsilon})$. By Remark~\ref{rem:dyn-prog-tractability}\,\emph{(iii)}, approximations of~$W_c(\mu,\nu_t)$ can therefore be computed in pseudo-polynomial time.
% {\color{black}
% For the discrete problem class defined in~Theorem~\ref{theorem:approximation},~\cite{genevay2016stochastic} propose to use a variant of stochastic gradient descent algorithm to find an~$\varepsilon$-approximate solution in~$\mc O(1/\epsilon^2)$ number of iterations where the guarantee holds in expectation. 
% Here we provide a dynamic programming based algorithm by exploiting the structure of the problem that approximates optimal transport problem in~$\mc O(1/ \varepsilon^4)$. 
% BT: I am not entirely sure how much information we should provide here. Should we mention that the special case of this algorithm would have a better convergence rate without proving it, only by words?
% }

%Our results illustrate that the discrete optimal transport problem can be challenging to solve when the underlying random vectors have independent marginals.  It would be interesting to develop computationally fast algorithms to solve this problem by exploiting the independence among the random variables. Another interesting extension is to identify other polynomial-time solvable instances.

\vspace{1ex}
\textbf{Acknowledgements.} This research was supported by the Swiss National Science Foundation under the NCCR Automation, grant agreement~51NF40\_180545. In addition, the second author was supported by an Early Postdoc.Mobility Fellowship, grant agreement P2ELP2\_195149, and the last author was partially supported by the MOE Academic Research Fund Tier 2 grant T2MOE1906.
{\color{black} We are grateful two anonymous referees, whose insightful comments led to significant improvements of the paper.}
\bibliographystyle{plainnat} 
\bibliography{references}

\end{document}